\documentclass{amsart}[12pt]

\usepackage{amsmath}
\usepackage{amsfonts}
\usepackage{amssymb}
\usepackage{amsthm}
\usepackage{amscd}
\usepackage{enumerate}
\usepackage{hyperref}
\newtheorem{thm}{Theorem}[section]
\newtheorem{lem}[thm]{Lemma}
\newtheorem{cor}[thm]{Corollary}
\newtheorem{prop}[thm]{Proposition}
\newtheorem{conj}[thm]{Conjecture}

\usepackage{todonotes}

\theoremstyle{definition}
\newtheorem{question}[thm]{Question}
\newtheorem{definition}[thm]{Definition}
\theoremstyle{remark} 

\newtheorem{example}{Example}

\newcommand{\mrm}{\mathrm}

\newcommand{\cay}{\mathrm{Cay}}
\newcommand{\id}{ (1) }
\newcommand{\fix}{\mathrm{fix}}
\newcommand{\Der}{\mathrm{Der}}

\newcommand{\Sym}{\mathrm{Sym}}
\newcommand\sym[1]{\Sym({#1})}
\newcommand{\Alt}{\mathrm{Alt}}
\newcommand\alt[1]{\Alt({#1})}
\newcommand{\pgl}{\mathrm{PGL}}

\newcommand{\mcal}{\mathcal}

\newcommand{\de}{\mathbf{d}}

\numberwithin{equation}{section}

\title[Robust EKR theorems on permutation groups]{A new measure of robustness of Erd\H{o}s--Ko--Rado Theorems on permutation groups}
\date{}

\author[Gunderson]{Karen Gunderson${^+}$}
\address{Department of Mathematics, University of Manitoba, Winnipeg, 
	Manitoba R3T 2N2, Canada}
\email{Karen.Gunderson@umanitoba.ca}
\thanks{${^+}$Research supported by the Natural Science and Engineering Research Council of Canada (grant RGPIN-2024-06029).}

\author[Meagher]{Karen Meagher${^*}$}
\address{Department of Mathematics and Statistics, University of Regina, Regina, SK, S4S 0A2, Canada}
\email[Corresponding author]{karen.meagher@uregina.ca}
\thanks{${^*}$Research supported by the Natural Science and Engineering Research Council of Canada (grant RGPIN-2018-03952).}

\author[Morris]{Joy Morris${^\dag}$}
\address{Department of Mathematics and Computer Science, University of Lethbridge, Lethbridge, Alberta T1K 3M4, Canada}
\email{joy.morris@uleth.ca}
\thanks{${^\dag}$Research supported by the Natural Science and Engineering Research Council of Canada (grant RGPIN-2024-04013).}

\author[Pantangi]{Venkata Raghu Tej Pantangi}
\address{Department of Mathematics and Statistics, University of Regina, Regina, SK, S4S 0A2, Canada}
\email{rpvrt1990@gmail.com}

\author[Shirazi]{Mahsa N. Shirazi}
\address{Department of Mathematics and Statistics, University of Manitoba, Winnipeg, 
	Manitoba R3T 2N2, Canada}
\email{mahsa.nasrollahi@gmail.com}

	\thanks{The authors are all indebted to the support of the Pacific Institute for Mathematical Sciences (PIMS), through the establishment of the Collaborative Research Group on Movement and Symmetry in Graphs. In connection with this CRG, this paper has the report identifier PIMS-202502DD-CRG36.}

\keywords{Erd\H{o}s-Ko-Rado Theorem, permutation groups, derangement graphs, Cayley Graphs Coclique}
\subjclass[2010]{05E30, 05C50, 05C25}
 		
\begin{document}

\begin{abstract}
In this paper we introduce a new way of measuring the robustness of Erd\H{o}s--Ko--Rado (EKR) Theorems on permutation groups. EKR-type results can be viewed as results about the independence numbers of certain corresponding graphs, namely the derangement graphs, and random subgraphs of these graphs have been used to measure the robustness of these extremal results. In the context of permutation groups, the derangement graphs are Cayley graphs on the permutation group in question. We propose studying extremal properties of subgraphs of derangement graphs, that are themselves Cayley graphs of the group, to measure robustness.

We present a variety of results about the robustness of the EKR property of various permutation groups using this new measure.
\end{abstract}
\maketitle


\section{Introduction}
\label{sec:Intro}

The Erd\H{o}s--Ko--Rado (EKR) Theorem~\cite{EKR1961} is a famous result in extremal set theory that gives the size of the largest possible collection of subsets, of a given cardinality from a fixed set, that are pairwise intersecting. It also characterises any such maximum collection. Specifically, if $\binom{[n]}{k}$ is the set of all $k$-subsets from $[n]= \{1,2,\dots,n\}$, the EKR theorem states that if $\mcal{F} \subset \binom{[n]}{k}$ is an intersecting family and $n\geq 2k$, then $|\mcal{F}| \leq \binom{n-1}{k-1}$. Moreover, if $n>2k$, this bound is achieved if and only if $\mcal{F}$ is the collection of all $k$-subsets that contain a fixed element from $[n]$. 

There is a general setup for ``EKR-type'' problems for a family $\mcal{G}$ of mathematical objects. Usually, it is possible to identify $\mcal{G}$ as a subset of the power set $2^{X}$ for some finite set $X$; the notion of intersection among objects in $\mcal{G}$ is equivalent to the intersection of subsets of $X$. A typical ``EKR-type'' problem asks to find the best upper bound on the size of an intersecting family, and often a natural version of the EKR Theorem holds on these objects. 
Given $x\in X$, the family $\mcal{G}_{x}:=\{A \in \mcal{G}\ :\ x\in A\}$, is an example of an intersecting family. The set $\mcal{G}_{x}$ is called the star centred at $x$. Because of their natural structure, stars are also referred to as \emph{canonical} intersecting sets.  In the classical EKR theorem, $X$ is the set $[n]$, the objects are all elements
of the power set $2^{X}$ with cardinality $k$, and the size of a star is $\binom{n-1}{k-1}$. The classical EKR theorem says for $n\geq 2k$, an intersecting set is no larger than a star.
 
A \emph{maximum} intersecting family is an intersecting set of the maximum possible size.  If a star is a maximum intersecting set in $\mcal{G}$, we say that $\mcal{G}$ satisfies the \emph{EKR property}. Moreover, if stars are the only maximum intersecting sets, then we say that $\mcal{G}$ satisfies the \emph{strict-EKR property}. So the classical EKR theorem is equivalent to saying if $n\geq 2k$, then the $k$-sets from $[n]$ have the EKR property, and if $n>2k$ they have the strict-EKR property.

A common way to consider an EKR-type problem is to rephrase it as a graph problem. We define a graph whose vertices are elements of $\mcal{G}$ in which two vertices are adjacent if and only if they are disjoint as subsets of $X$. Such a graph is called the \emph{derangement graph on $\mcal{G}$}. We observe that intersecting families in $\mcal{G}$ are exactly the independent sets (also called cocliques) in the derangement graph on $\mcal{G}$. As usual, the size of the largest independent set in a graph $\Gamma$ will be denoted by $\alpha(\Gamma)$ and is also called the \emph{independence number} of the graph. In the case of the classical EKR theorem, as noted above, $\mcal{G}=\binom{[n]}{k}$. The derangement graph in this case is the Kneser graph $K(n,k)$. The original EKR theorem translates to the statement that the largest independent set in the Kneser graph $K(n,k)$ has size $\binom{n-1}{k-1}$, or
\[
\alpha(K(n,k)) =  \binom{n-1}{k-1}.
\]
Provided that $n >2k$, this bound is only achieved by the sets of all $k$-subsets that contain a fixed point; these are the \emph{stars} in this setting.

There has been a large amount of work that considers the case where $\mcal{G}$ is the set of permutations in a given permutation group~\cite{CameronKu, li2020erd, MR3780424, MR2739502,meagher2016erdHos,MeagherSin,spiga2019erdHos}, and this is the family of mathematical objects we focus on in this paper.        

Let $G \leq \sym{n}$ be a permutation group of degree $n$. Here $\mcal{G}$ is the set of elements in $G$, and $X$ is the set of ordered pairs $(i,j)$ with $i,j \in [n]$. Each element $\sigma \in G$ is identified with the subset $\{ (i,\sigma(i)) \ :\ i\in [n] \}$ from $2^X$. Two permutations $\pi, \sigma \in G$ are \emph{intersecting} if there exists an  $i \in \{1 \dots n\}$ with $\pi(i) = \sigma(i)$ (this is equivalent to intersection in $2^X$). 

Now we can formulate an EKR-type problem. For any permutation group we can ask what is the largest set of elements that pairwise intersect. If $G$ is transitive, then the \emph{star} centred at any $x = (i,j)$, denoted by
\[
G_{i\to j}  = \{ \sigma \in G \, | \, \sigma(i) = j \},
\]
is an intersecting set of size $|G|/n$. Note that a star is either the stabilizer of a point or a coset of the stabilizer of a point. As defined above, if a star is a maximum intersecting set in $G$, then $G$ is said to have the EKR property. 

The derangement graph for a permutation group $G$ is the graph with the elements of $G$ as its vertices, and two vertices are adjacent if and only if the elements are not intersecting. Observe that the intersecting sets in $G$ are exactly the independent sets in this graph. Two elements in $\sigma, \pi \in G$ are adjacent in the derangement graph exactly when $\sigma ^{-1}\pi$ has no fixed points, so when $\sigma ^{-1}\pi$ is a \emph{derangement}. This means that the derangement graph for a group is a \emph{Cayley graph}.

Let $S \subset G$ with $S=S^{-1}$ and the identity, which we denote by $\id$, not in $S$. Recall that a Cayley graph, denoted $\cay(G,S)$, is the graph whose vertices are the elements of $G$, with $g$ adjacent to $gs$ if and only if $s \in S$. (The condition $S=S^{-1}$ ensures that this produces a graph rather than a digraph, and the requirement that $\id \notin S$ avoids loops in the Cayley graph.) Equivalently, $g, h \in G$ are adjacent in $\cay(G,S)$ if and only if $g^{-1}h \in S$. Cayley graphs are often defined with adjacency determined by multiplication by elements of $S$ on the left rather than the right, and analogous results hold under either definition. We refer to $S$ as the \emph{connection set} of the Cayley graph $\cay(G,S)$. For the derangement graph for any group $G$, the connection set is the collection of all derangements, denoted $\Der(G)$. We denote the derangement graph on $G$ by
\[
\Gamma_G = \cay(G, \Der(G) ).
\]
  
When a family $\mcal{G}$ satisfies the strict-EKR property, we can also consider the ``stability" or ``robustness'' of the characterization of maximum intersecting sets in $\mcal{G}$, and this can be measured in different ways. There has been much work that demonstrates different types of robustness of the original EKR theorem. For example, an old result in this area is the Hilton-Milner theorem~\cite{MR0219428} that bounds the size of an intersecting set that is not contained in a star. The Hilton-Milner theorem shows that such a set is much smaller than $\binom{n-1}{k-1}$, so the stars are not only the largest intersecting sets, but there is no other maximal intersecting family that is even close in size to the stars. Analogous results for some permutation groups can be found in the literature, for instance in \cite{Ellis2012, Pan2023, Plaza2015}. 

A different way of exhibiting robustness is to look at the independence number of a random spanning subgraph. For instance, recent results have determined the threshold probability $p_c(n, k)$ for the event that when edges of $K(n,k)$ are retained independently at random with a given probability,  that the size of the largest independent set in the resulting graph will still be no larger than $\binom{n-1}{k-1}$~\cite{BBN2015, BKL2023, BNR2016, DK2016, DT2015}.

There are similar results for groups. For example, in~\cite{CRGRandomPaper} (by the same authors as this paper) the critical probability is determined for the event that, when edges are removed randomly from $\Gamma_{\sym{n}}$,  the resulting graph will not have larger independent sets than the original derangement graph. But this approach of simply removing edges at random from the derangement graph felt less natural than the corresponding operation in the Kneser graph, as it broke important symmetry: the resulting graph would almost certainly not remain vertex-transitive, or have a nice relationship to any action of the group $G$. 
 So rather than removing edges from $\Gamma_G$ at random, it makes sense to consider removing a derangement, and its inverse, from the connection set and looking at the resulting Cayley graph. Removing edges in this way is the only way to produce a subgraph that remains a Cayley graph on the group $G$. This seems an important consideration when what we are studying is the structure of intersecting elements of the group $G$. As noted previously, the original context of intersecting sets is modelled by the Kneser graph. The Kneser graph is not a Cayley graph, so a similar approach is not possible in that context. 
 
 When studying threshold probabilities the event that the independence number increases when edges of a graph are deleted at random, extremal questions come up naturally in the context of first and second moment method arguments.  One uses estimates on the number of edges spanning a set of a given size and how many sets could span a minimum number of edges.
 
In this paper, we consider some transitive groups $G\leq \sym{n}$ that satisfy the EKR property (equivalently  $\alpha(\Gamma_{G})=|G|/n$). For these groups, we attempt to answer the question: \emph{For which inverse-closed subsets $D\subset \Der(G)$, is 
\[
\alpha\left(\cay(G,\Der(G)\setminus D) \right) > \alpha(\Gamma_{G})?
\]} 
Furthermore, which are the \emph{smallest} sets whose deletion from the connection increases the independence number.
There is a natural way to label edges of $\Gamma_{G}$: for all $g \in G$ and $d \in \Der(G)$, the edge connecting $g$ with $gd$ is assigned $\{d,\ d^{-1}\}$---we call these pairs \emph{labels} (a label contains only one element if and only if $d$ is self-inverse). We note that, for an inverse-closed subset $D\subset \Der(G)$, the graph $\cay(G,\ \Der(G)\setminus D)$ is the spanning subgraph of $\Gamma_{G}$  which is obtained by removing edges with labels in $\{ \{d,d^{-1}\}\ :\ d\in D\}$. For brevity, we shall refer to process of deleting edges labeled $\{d,\ d^{-1}\}$ as removing the \emph{label} $\{d,\ d^{-1}\}$ from the connection set. Removing a single label corresponds to removing a $2$-factor (so a set of $|G|$ edges) from the graph $\Gamma_G$, or a perfect matching (which is a set of $|G|/2$ edges) if the removed derangement is self-inverse.

There is a simple way to construct a set $D$ for any permutation group $G\leq\sym{n}$ so that $\alpha(\cay(G, \Der(G) \setminus D)) >\alpha(\Gamma_G)$. Given  $i,j \in [n]$, set $D_{i\to j}:=G_{i\to j} \cap \Der(G)$. Proposition~\ref{prop:binarystars} shows that for any $i,j \in [n]$ with $i\neq j$, the set $G_{j \to j} \cup G_{i\to j}$ will be an independent set of size $2 \alpha(\Gamma_G)>\alpha(\Gamma_G)$ in  
\[
\cay\left(G,\ \Der(G) \setminus (D_{i\to j} \cup D_{j\to i}) \right).
\] 
We will refer to sets of the form $G_{j \to j} \cup G_{i \to j}$ as \emph{binary stars}.
Given a transitive permutation group $G$, we define $\de_{G}$ to be the minimum number of distinct labels in $D_{i \to j}$. This is the number of distinct sets $\{d,d^{-1}\}$ with either $d$ or $d^{-1}$ in $D_{i \to j}$.
\begin{equation*}
\de_{G} :=\mrm{min}_ {i,j \in [n], \ i \neq j } \left |  \left\{  \{d,d^{-1}\} \ : \   d \in D_{i \to j} \right\}  \right|.
\end{equation*}

It is now natural to ask whether there is more ``efficient'' way of constructing a Cayley subgraph of $\Gamma_{G}$ whose independence number is larger than $\alpha(\Gamma_{G})$ than by creating binary stars; that is, a way to remove fewer labels but still increase the independence number. Perhaps surprisingly, we have been able to identify many situations in which there is no more efficient way of increasing the independence number. This (together with related ideas) is the primary measure of robustness that we study in this paper.
\begin{definition}\label{def:robust}
A transitive permutation group $G$ is called \emph{EKR robust} if 
\[\alpha(\cay(G, \Der(G) \setminus D))= |G|/n, \] for all inverse-closed subsets $D \subset \Der(G)$ with $|\{\{d, d^{-1}\} \ :\  d \in D\}|< \de_{G}$; that is, inverse-closed subsets $D$ that contain fewer than $\de_G$ labels.
\end{definition}

Any $G$ that does not have the EKR property is also not EKR robust since $\alpha( \cay(G, \Der(G))  ) >  |G| /n$. On the other end of things, any group $G$ that satisfies the EKR property and has $\de_{G}=1$ must be EKR robust. We now mention some examples of EKR robust groups that we show fall into this category.

\begin{example}\label{ex:ekrrobustexamples1}

\begin{enumerate}
\item If $G\leq \sym{n}$ is a regular permutation group, as we observe in Section~\ref{sect:cliques}, $G$ satisfies the EKR property and  $\de_{G}=1$. Thus all regular permutation groups are EKR robust. In particular, since every transitive abelian permutation group is regular, all of these groups are EKR robust.
\item If $G$ is a Frobenius group, then in  Section~\ref{sect:cliques} we observe that $G$ satisfies the EKR property and has $\de_G=1$, so is EKR robust.
\item Generalized dihedral groups are EKR robust. If they are acting regularly, they are covered above. They have only one possible non-regular transitive permutation representation, and in this situation we show in Theorem~\ref{thm:gendi} that again $\de_{G}=1$.
\end{enumerate}
\end{example} 


While it is always possible to pick $\de_{G}$ labels that can be removed to increase the independence number of $\Gamma_{G}$, it is not typically the case that removing any choice of $\de_G$ labels will increase the independence number. It can also be possible to choose significantly more than $\de_{G}$ labels to remove, without increasing the independence number. One such example from the literature on EKR theorems is given in~\cite[Section 6]{meagher2016erdHos} where 3 entire conjugacy classes of derangements are removed from the connection set for the derangement graph of the Higman-Sims group without changing the size of the maximum independent sets at all. In this example, the number of derangements removed is 5,902,050, representing the removal of 2,951,025 labels (since none of the derangements that were removed are self-inverse). The total number of derangements in the group is 13,960,050, and since it is $2$-transitive of degree $176$, this means that $13,960,050/175=79,806$ of these derangements map any fixed $i$ to any fixed $j$, so $\de_G \leq 79,806$.

 For generalized dihedral groups, we study in more depth how the independence number can vary when we remove labels from the derangement graphs. We show that removing a single label either leaves the independence number unchanged, or doubles or triples it. We characterise exactly when each of these situations occurs.  
We delve into this question in even more depth in the case of dihedral groups, showing that after removing $2$ labels the resulting graph must have an independence number that is even and at most 10, and characterising exactly how each possible value can arise.

We also show that if the abelian index-2 subgroup of the generalised dihedral group includes any odd permutations, then removing any set of labels  will not result in a subgraph of $\Gamma_{G}$ whose independence number is $3$; but, if this subgroup contains only even permutations, then there is a set of labels that can be removed to produce a subgraph of independence number $3$. Our interest in the independent sets of size $3$ comes from the fact that it is the only value $k$ such that $\alpha(\Gamma_G)<k<2\alpha(\Gamma_G)$, so these independent sets are larger than stars, but are not the union of two stars (binary stars are a special case of the union of two stars). We also show that when there are odd permutations in the abelian index-2 subgroup, then similar to the previous case, removing all odd derangements does not increase the independence number---but in this situation, the independence number remains at $\alpha(\Gamma_G)$, not $2\alpha(\Gamma_G)$. These results provide some insight into what independence numbers can arise and how the removal of very few labels can double the size of the maximum independent set. 

In Section~\ref{sect:other groups} we give another example of an EKR robust group: the $2$-transitive group $\pgl(2,q)$. We use similar arguments to show that for a symmetric group of prime degree $p$, at least $(p-2)!$ 
labels must be removed from $\Gamma_G$ before the value of the independence number changes, but this is less than $\de_G$. We present some results about symmetric groups of small degree: in particular, we show that $\sym{5}$ is not EKR robust.
In Section~\ref{sect:subgroups} we present some results about subgroups $H \le G$, and how understanding the independence number of $\Gamma_H$ can give us information about $\Gamma_G$.

We conclude the paper with questions and open problems for future work. There is still much work to be done on this topic and we find this to be a very appealing problem. As mentioned previously, one of the motivations for studying this problem arose from taking random subgraphs of derangement graphs associated with groups that have the EKR property.  We believe it is another natural direction to investigate would be to look at the independence number of Cayley graphs whose connection set is obtained by randomly deleting derangements and their inverses from the connection set.  The extremal results that we prove here and those that remain for future work would serve as key tools in any results on these random Cayley graphs. Studying the extremal properties of the connection sets may also provide a way to refine the definition of the EKR property for groups and to better understand EKR properties for groups.

\section{Background}\label{sect:back}
 
For any transitive group $G$, the star $G_{i\to j}$ is an independent set in $\Gamma_G=\cay(G,\Der(G))$ with size $|G|/n$. 
In this paper we only consider transitive groups, and we will focus on the size of an independent set in $\cay(G,\Der(G) \setminus D)$ for different $D \subseteq \Der(G)$.
We mentioned earlier, without proof, that removing the labels corresponding to elements of $D_{i \to j}$ from $\Der(G)$ leaves an independent set of size at least $2\alpha(\Gamma_G)=2|G|/n$ in the resulting subgraph (there is then an independent set containing what we have called a binary star). We now prove this.

\begin{prop}\label{prop:binarystars}
Let $i, j$ be distinct fixed elements in $[n]$ and let $D_{i\to j}$ be the set of all derangements in a transitive group $G \leq \sym{n}$ that map $i$ to $j$. Consider $\cay(G, \Der(G) \setminus D )$, where $D = D_{i\to j} \cup D_{j \to i}$ (so all the permutations that map $i$ to $j$ are removed from the connection set of $\Gamma_G$, together with their inverses). In $\cay(G,  \Der(G) \setminus D )$, the set $G_{j \to j}\cup G_{i \to j}$ is an independent set of size $2|G|/n$. 
\end{prop}
\begin{proof}
Clearly $G_{j\to j}$ and $G_{i\to j}$ are disjoint independent sets (in fact, stars) in $\cay(G, \Der(G))$ and each has size $|G|/n$. To show that $G_{j\to j} \cup G_{i\to j}$ is an independent set in $\cay(G,  \Der(G) \setminus D )$ we need to show that there are no edges between any vertex in $G_{j \to j}$ and any vertex in $G_{i \to j}$. If $\sigma \in G_{j\to j}$ and $\pi \in G_{i\to j}$, then $ \sigma^{-1} \pi (i) = \sigma^{-1} (j) = j$. Thus either $\sigma^{-1}\pi \in D_{i \to j}$, or $\sigma^{-1}\pi$ is not a derangement; in either case, $\sigma^{-1}\pi \notin \Der(G)\setminus D$. Therefore $\sigma$ and $\pi$ are not adjacent in $\cay(G, \Der(G) \setminus D)$. 
\end{proof}

The following theorem is the well-known \emph{clique-coclique bound}. It is an effective method for determining the independence number of some Cayley graphs. We use $\omega(X)$ to denote the size of the largest clique in $X$.

\begin{thm}[Clique-coclique bound]\label{clique-coclique}
Let $\Gamma$ be a vertex-transitive graph. Then 
\[
\alpha(\Gamma) \, \omega(\Gamma) \leq |V(\Gamma) |.
\]
\end{thm}

For any (inverse-closed) connection set $C$, the graph $\cay(G, C)$ is vertex-transitive. If $\cay(G, C)$ contains a clique of size $k$, then by Theorem~\ref{clique-coclique}, $\alpha(\cay(G, C)) \le |G|/k$. In particular, this arises in the following situation (note that in this result we do not need to assume that $C \subset \Der(G)$).

\begin{prop}\label{subgroup-bound}
Let $G$ be a group. If $C \subseteq G$ is such that $C \cup \{\id\}$ includes a subgroup $H$ of $G$, then $\alpha( \cay(G, C) ) \le |G|/|H|$.
\end{prop}
\begin{proof}
The set $H$ induces a clique of size $|H|$ in $\cay(G, C)$ and $\cay(G, C)$ is vertex-transitive, so the result follows by Theorem~\ref{clique-coclique}. 
\end{proof}

\begin{cor}\label{cor-subgroup-bound}
Let $G \le \sym{n}$ be a group and $\Der(G)$ the set of all derangements in $G$. If $C \subseteq \Der(G)$ is such that $C \cup \{\id\}$ includes a transitive subgroup $H$ of $G$, then $\alpha( \cay(G, C) ) = |G|/n$.
\end{cor}
\begin{proof}
Assume that $H \subseteq C \cup \{\id\}$ is a transitive subgroup, then $|H| \ge n$, and Proposition~\ref{subgroup-bound} implies $\alpha(\cay(G,C)) \le |G|/n$. Since $H$ is transitive, $G$ is also transitive so the stars in $\Gamma_G$ are independent sets with $|G|/n$ vertices. Thus we conclude that $\alpha(\cay(G,C))=|G|/n=\alpha(\Gamma_G)$. 
\end{proof}

This corollary implies any group $G$ for which $\Der(G)\cup \{ \id \}$ contains a transitive subgroup has the EKR property. Further, if we are to find a Cayley subgraph  $\cay(G, C) \subset \Gamma_G$ that has a larger independent set, then we need $C\cup \{ \id \}$ not to contain any transitive subgroup of $G$. One way to be certain that $C \cup\{\id\}$ does not contain a transitive subgroup $H$ of $G$, is to remove all derangements that map $i$ to $j$. This is exactly the situation we have considered in Proposition~\ref{prop:binarystars}, which results in $\cay(G,C)$ having a binary star.

To end this section we need to make some remarks about group actions. Clearly whether or not a group has the EKR property depends on which action is considered. We only consider transitive actions, and it is well-known that any transitive action of a group $G$ is equivalent to its action on cosets $G/H$ for some subgroup $H$. With this representation it is easier to determine the elements which are not derangements.

\begin{prop}\label{prop:representation}
Let $H \leq G$ be groups. Under the action of $G$ on $G/H$, an element of $G$ has a fixed point if and only if it is conjugate to an element in $H$.
\end{prop}
\begin{proof}
An element $g \in G$ has a fixed point, if and only if $g yH = yH$ for some coset $yH$ of $H$ in $G$.
This happens exactly when $y^{-1} g y \in H$; that is, when $g$ is conjugate to an element in $H$.
\end{proof}

\section{Cliques in $\Gamma_G$ from subgroups of $G$}\label{sect:cliques}

Some of the ideas about cliques that were presented in Section~\ref{sect:back} can be formulated into a general statement that proves surprisingly useful, but applies also in the situation where a subgroup $H$  of $G$ contained in $\Der(G)\cup\{ \id \}$ is intransitive. In fact, for the following to apply (which requires $\alpha(\Gamma_G)=|G|/|H|$), we must have that the size of $H$ is no larger than the degree of $G$. This result shows essentially the opposite of robustness: in this situation, there is a single label that can be removed, with the result that the independence number at least doubles. 

\begin{prop}\label{2xsize}
Let $G$ be a group, and suppose $H \subseteq \Der(G)\cup \{\id\}$ is a subgroup of $G$ with $\alpha(\Gamma_G)=|G|/|H|$.
Further suppose that there exists a set $S$ and an element $h \in H$ such that:
\begin{enumerate}[i.]
\item  $S$ is an independent set in $\Gamma_G$;
\item  $|S|=\alpha(\Gamma_G)$; and 
\item for every distinct $s, t \in S$, $s^{-1}t h \notin \Der(G)$.
\end{enumerate}
Then taking $D= \{h,h^{-1}\}$, we have $S\cup Sh$ is an independent set in $\cay(G,\Der(G)\setminus D)$, and therefore $\alpha(\cay(G, \Der(G)\setminus D ))\ge 2\alpha(\Gamma_G)$.
\end{prop}

\begin{proof}
Since $S$ is an independent set in $\Gamma_G$, for $s, t \in S$, $s^{-1}t$ is not a derangement. This also implies there is no edge between $sh$ and $th$, since $(sh)^{-1}th=h^{-1} s^{-1}t h$ is conjugate to $s^{-1}t$, which is not a derangement. Thus, it suffices to show that for any two elements one in $S$ and the other in $Sh$, there is no edge between them.

There is an edge between $s$ and $th$, where $s, t \in S$, if and only if $s^{-1} th \in \Der(G)\setminus D $. If $s\neq t$ then this is not the case since $s^{-1} th \notin\Der(G)$ by hypothesis. If $s=t$ then $s^{-1}th=h \notin  \Der(G)\setminus D $ by definition of $D$.
\end{proof}

We use this result to consider the situation where the derangement graph is a disjoint union of cliques. 
Note that when $\Gamma_G$ is a disjoint union of more than two cliques larger than $K_1$, stars are very far from being the only maximum independent sets, since taking any single vertex from each clique yields an independent set of maximum size. In the case where the derangement graph is a union of only two cliques, any element $\sigma$ from the first clique must agree on some point with any element $\pi$ from the second clique, since otherwise $\sigma^{-1}\pi$ would be a derangement and the derangement graph would not be a disjoint union of the two cliques. However, as soon as a third clique is introduced, there is a maximum independent set $\{\pi, \sigma, \tau\}$ with each element from a different clique. Both $\tau^{-1}\pi$ and $\tau^{-1}\sigma$ must have fixed points, but we can choose these to be distinct, so that the three permutations do not lie together in a star. So groups whose derangement graph is a disjoint union of at least three cliques (but is not the empty graph) have the EKR property, but not the strict-EKR property (which requires that all maximum independent sets are stars).
If $\Gamma_G$ is either a clique, the empty graph, or the union of only two cliques, then $G$ has both the EKR property and the strict-EKR property. 

We recall that a transitive permutation group $G$ is \emph{regular} if it is transitive and no non-identity element fixes a point. 

\begin{lem}\label{lem:unioncliques}
If $H=\Der(G)\cup\{\id\}$ forms a subgroup of a  $G \le \sym{n}$, then 
$\Gamma_G =  {   \bigcup_{i=1}^{m} K_{|H|} }$
where $m = |G:H|$. In fact, each $K_{|H|}$ contains all the elements in an $H$-coset in $G$. Further, $\de_G\le 1$, and when $G$ is transitive, $H$ is regular and $\de_G=1$. Hence any transitive group $G$ with this property has the EKR property and is EKR robust.
\end{lem}
\begin{proof}
Distinct vertices $g, h \in G$ are adjacent in $\Gamma_G$ if and only if $g^{-1}h \in H$. This happens exactly when $g$ and $h$ are in the same coset of $H$ in $G$, so $\Gamma_G =  {   \bigcup_{i=1}^{m} K_{|H|} }$.

Suppose $\de_G > 1$, so there are derangements $h_1, h_2 \in \Der(G) \subset H$, and distinct $i, j \in [n]$, such that $h_1(i)=h_2(i)=j$. But this implies $h_2^{-1}h_1(i)=i$ so $h_2^{-1}h_1 \in H \setminus \Der(G)=\{1\}$, so $h_1=h_2$. Thus $\de_G\le 1$.

Now, we assume that $G$ is transitive. If $H$ is intransitive, the action of $H$ has at least two orbits. The transitivity of the action of $G$ on $[n]$ implies that $G$ acts transitively on the set $O$ of orbits of $H$. A standard counting argument in group theory implies that every transitive group of degree at least two contains a derangement, so there is some $g \in G$ that fixes no point of $O$. But this forces $g$ to be a derangement on the elements of $[n]$, so $g \in H$, a contradiction. Thus $H$ is transitive, as claimed. The transitivity of $H$ implies the EKR property by Corollary~\ref{cor-subgroup-bound}.

Assuming $H$ is transitive, for every distinct $i, j \in [n]$, there is some $h \in H$ such that $h(i)=j$. Since $i \neq j$ we have $h \neq 1$, so $h \in \Der(G)$. This implies that $H$ is regular, and $\de_G \ge 1$. We conclude that $\de_G=1$ and as previously noted, this trivially implies that $G$ is EKR robust. 
\end{proof}

Given a group $G$ satisfying the premise of Lemma~\ref{lem:unioncliques}, we now find the precise formula for the independence number of the subgraphs of $\Gamma_{G}$ which are obtained by removing a single label.  This implies a different version of robustness: for the derangement graph of these groups, it is not possible to obtain a Cayley subgraph whose independence number lies strictly between the independence number of $\Gamma_G$, and double that value: that is, the same independence number we get if we create a binary star. So not only do we need to remove at least $\de_G$ labels to produce an increase in the independence number, but any increase we get cannot be smaller than the increase we would get by removing all elements of $D_{i \to j}$.

In the following result and elsewhere, we use $o(h)$ to denote the order of an element $h$ in a group.

\begin{cor}\label{cor:subgroup-der}
Suppose that $G$ acts transitively and $H=\Der(G)\cup\{\id\}$ forms a subgroup of $G$.
Then for any $h \in \Der(G)$,
if $D= \{h,h^{-1}\}$, then 
\[
\alpha(\cay(G, \Der(G)\setminus D)) =
\begin{cases} 
  2\alpha(\Gamma_G), &\text{if } o(h) \neq 3;\\ 
  3\alpha(\Gamma_G), &\text{if } o(h) = 3.
\end{cases}
\] 

In particular, there is no inverse-closed $D \subset \Der(G)$ such that 
\[
\alpha(\Gamma_G) < \alpha(\cay(G, \Der(G)\setminus D)) <2 \alpha(\Gamma_G).
\]
\end{cor}

\begin{proof}
By Lemma~\ref{lem:unioncliques}, $\alpha(\Gamma_G)=|G|/n$, and $H$ is regular, so $|H|=n$. Let $S$ be an independent set of size $\alpha(\Gamma_G)=|G|/n=|G|/|H|$ in $\Gamma_G$. We will apply Proposition~\ref{2xsize} to $S$ with $h \in \Der(G)$ arbitrary. 

Take any distinct $s, t \in S$, we claim  $s^{-1}th \notin \Der(G)$. To see this, note that by definition of $S$ and $H$, the elements $s$ and $t$ must lie in distinct left cosets of $H$. If $s^{-1}th \in \Der(G)$, then $s^{-1}th \in H$, so $s^{-1}t \in H$, which contradicts $s$ and $t$ being in distinct left cosets of $H$.  Thus, Proposition~\ref{2xsize} tells us that $S \cup Sh$ is an independent set and $\alpha(\cay(G,\Der(G)\setminus D))\ge 2\alpha(\Gamma_G)$.

We will next show that there is no independent set larger than $S \cup Sh$ unless $o(h)=3$, in which case $S \cup Sh \cup Sh^{-1}$ is the largest possible independent set. Let $T$ be an independent set in $\cay(G, \Der(G)\setminus D)$ of maximum possible size, and for any left coset $sH$, let $t \in sH\cap T$. Then $t$ is adjacent to every vertex of $sH$ except $th$ and $th^{-1}$. So $T$ can have at most $3$ vertices from any left coset of $H$, and only $2$ if $\o(h) =2$. This shows if $\o(h) >2$, then $\alpha(\cay(G,\Der(G)\setminus D)) \le 3|G|/|H|=3 \alpha(\Gamma_G)$. Furthermore, if $o(h) > 3$ then $h^2 \in H$ is a derangement, distinct from $h$ and $h^{-1}$, and is in $\Der(G)\setminus D$, so $th^{-1}$ is adjacent to $th$ in $\cay(G, \Der(G)\setminus D)$. Thus if $o(h) \neq 3$ then $T$ can have at most $2$ vertices in any left coset of $H$, so $|T|\le 2\alpha(\Gamma_G)$, meaning $S\cup Sh$ is a largest possible independent set. 

To complete the proof we need only show that if $o(h)=3$ then $S \cup Sh \cup Sh^{-1}$ is an independent set. By Proposition~\ref{2xsize}, both $S \cup Sh$ and $S \cup Sh^{-1}$ are independent sets. So we only need to show there are no edges between vertices of $Sh$ and $Sh^{-1}$; this amounts to showing that $h^{-1}s^{-1}th^{-1} \notin \Der(G)\setminus D $ for all $s,t \in S$. Since $o(h)=3$ we have $h^{-1}s^{-1}th^{-1}=h^{-1}(s^{-1}th)h$. If $s=t$, then $h^{-1}(s^{-1}th)h = h \notin \Der(G)\setminus D$. If $s \neq t$, then $s$ and $t$ belong to different cosets of $H$, so $s^{-1}th \not \in H$, and as $h^{-1}(s^{-1}th)h$ is conjugate to $s^{-1}th \not \in H$, it is also not a derangement. 

The final statement follows since for any inverse-closed proper subset $D$ of $\Der(G)$, there is some $h \in \Der(G)$ such that $\Gamma_2:=\cay(G,D)$ is a subgraph of $\Gamma_1:=\cay(G, \Der(G)\setminus\{h,h^{-1}\})$. Therefore $\alpha(\Gamma_2) \ge \alpha(\Gamma_1)$, which we have just shown is at least $2 \alpha(\Gamma_G)$.
\end{proof}

This covers a surprising number of groups, including regular groups and Frobenius groups, as mentioned in Example~\ref{ex:ekrrobustexamples1}.

\begin{cor}\label{regular}
Any group $G \leq \sym{n}$ whose action is regular, as well as any Frobenius permutation group $G$, satisfies the hypotheses and therefore the conclusions of  Lemma~\ref{lem:unioncliques} and Corollary~\ref{cor:subgroup-der}:
\begin{enumerate}[i.]
\item For any $h \in \Der(G)$, if we take $D= \{h,h^{-1}\}$, then 
\[
\alpha(\cay(G, \Der(G)\setminus D)) =
\begin{cases} 
2\alpha(\Gamma_G), &\text{if } o(h) \neq 3; \\ 
3\alpha(\Gamma_G), &\text{if } o(h) = 3 
\end{cases}.
\]
\item There is no inverse-closed $D \subset \Der(G)$ such that 
\[
\alpha(\Gamma_G)<\alpha(\cay(G, \Der(G)\setminus D)) <2 \alpha(\Gamma_G).
\]
\item The group has the EKR property and is EKR robust.
\end{enumerate}
\end{cor}

\begin{proof}
If the action of $G$ is regular, then no nontrivial element fixes any point, so every nontrivial element is a derangement. In this case, $H=\Der(G)\cup \{ \id \}=G$.

If $G$ is a Frobenius group, then it is well-known that $\Der(G)\cup\{ \id \}$ is a normal subgroup of $G$, called the Frobenius kernel of $G$.

By assumption the group $G$ is transitive, so Lemma~\ref{lem:unioncliques} and Corollary~\ref{cor:subgroup-der} imply (ii) and (iii).
\end{proof}

Having observed that this covers all transitive abelian groups, we further demonstrate the importance of this result by also pointing out it is known that for any generalised dicyclic group, the only faithful transitive representation is regular. Let $A$ be a cyclic group of even order and $y$ an involution (element of order $2$) in $A$, then the generalised dicyclic group over $A$ and $y$ is the group
\[
\textrm{Dic}(A,y)=\langle x, A \ : \ x^2=y, \ x^{-1}ax=a^{-1} \ \textrm{for all } a \in A\rangle.
\]
In $\textrm{Dic}(A,y)$, every element of $xA$ has order $4$, since for every $a \in A$, $(xa)^2=x^2=y$. 
 
 \begin{prop}\label{gen-dicyclic-regular}
 The only faithful transitive representation of a generalised dicyclic group is the regular representation.
 \end{prop}
 \begin{proof}
 In any generalised dicyclic group $G=\textrm{Dic}(A,y)$, every subgroup of $A$ is normal in $G$. First we will show this implies the action of $A$ is semiregular (any subgroup of $A$ that fixes some point of the underlying set fixes every point of the underlying set).
 
Let $H$ be the subgroup of $A$ that fixes some point $u$ of the underlying set. For $h \in H$, consider $h(v)$, where $v$ is any point of the underlying set. Since $G$ is transitive, there is some $g \in G$ such that $g(u)=v$. Now $g^{-1}h(v)=g^{-1}hg(u)=h'(u)$ for some $h' \in H$ as $H$ is normal in $G$. By definition of $H$, $h'(u)=u$. So $h(v)=g(u)=v$. This implies that $h$ fixes every point of the underlying set, since this representation is faithful, this implies the action of $A$ is semiregular. 

Using the orbit-stabiliser theorem, we conclude that the length of any orbit of $A$ is $|A|$. Now the fact that $A$ has index $2$ in $G$ and $G$ is transitive, implies that the degree of $G$ is either $|A|$ or $2|A|=|G|$. 
In the latter case, $G$ must be acting regularly, as claimed. In the former case, the orbit-stabiliser theorem tells us that the stabiliser in $G$ of any point has order $2$, so is generated by an involution that must lie in $xA$. Since there are no involutions in $xA$ in $G$ (as noted above, every element has order $4$), this completes the proof.
\end{proof}

There is one more interesting aspect to the robustness of the independence number of graphs $\Gamma_G$ that are disjoint unions of complete graphs. Essentially, the following result says that if any of the derangements of $G$ are odd permutations, then after doubling the cardinality of the maximum independent set by removing one label, it is possible to remove every label that is an odd permutation before the cardinality of the maximum independent set changes again. Note that there are almost certainly other ways of removing far fewer labels that do further increase the independence number, so this does not imply robustness in our usual sense.

\begin{prop}\label{robust}
Let $G \le \sym{n}$ and suppose that $\Der(G)\cup \{ \id \} \le G$ with $\Der(G) \not\subset \alt{n}$. 
If $D \subseteq \Der(G) \setminus \alt{n}$ (and $D$ is not empty), then
\[
\alpha(\cay(G, \Der(G) \setminus D ))=2\alpha(\Gamma_G).
\]
\end{prop}

\begin{proof}
Let $H=\Der(G)\cup \{\id\}$.
Since $D \subset \Der(G)$, Corollary~\ref{cor:subgroup-der} implies that $\alpha(\cay(G, \Der(G) \setminus D  ))\ge 2\alpha(\Gamma_G)=2|G|/|H|$. As $H\cap \alt{n}$ is a subgroup of order $|H|/2$, the vertices corresponding to this subgroup induce a clique of order $|H|/2$ in $\cay(G, \Der(G) \setminus D  )$. Thus the clique-coclique bound (Theorem~\ref{clique-coclique}) implies 
\[
\alpha(\cay(G, \Der(G) \setminus D  )) \le |G|/(|H|/2)=2|G|/|H|=2\alpha(\Gamma_G). \qedhere
\]
\end{proof}

Note that if $H$ has an element of order $3$ then, since it is a derangement, its disjoint cycle decomposition is entries made of length $3$ cycles, so it is an even permutation. So the preceding result does not contradict previous results in this section.


 \section{Dihedral and Generalised Dihedral Groups}\label{sect:dihedral}

Given any abelian group $A$ with $|A|=n$, the generalised dihedral group is the group 
\[
D(A)=\langle x, A \ : \ x^2=\id, \ xax=a^{-1} \ \textrm{ for all } a \in A\rangle.
\]
Since it is well-known that $D(A)$ is abelian if and only if $A$ is an elementary abelian $2$-group, we further assume that $A$ is not an elementary abelian $2$-group. 

\begin{prop}
The group $D(A)$ only has two faithful transitive permutation representations: the regular action, and the action of $D(A)$ on the coset $D(A)/H$ where $H=\langle x \rangle$.
\end{prop}
\begin{proof}
Let $\Omega$ be a set on which $D(A)$ acts faithfully and transitively.  Since every subgroup of $A$ is normal in $D(A)$, exactly as in the proof of Proposition~\ref{gen-dicyclic-regular}, this implies that $A$ must act semiregularly on $\Omega$. Hence $|\Omega|$ must be a multiple of $|A|$ as well as a divisor of $|D(A)|$, and thus is either $|A|$ or $|D(A)|=2|A|$. If $|\Omega|=2|A|$, then the action of $D(A)$ on $\Omega$ must be equivalent to the regular action of $D(A)$.

If $|\Omega|=|A|$, then by the orbit-stabilizer theorem, the stabilizer of a point must be a non-normal subgroup of order $2$. All non-normal subgroups of order $2$ are conjugate to $H:=\langle x \rangle$. 
Therefore, every non-regular transitive permutation action of $D(A)$ is equivalent to the action of $D(A)$ on $D(A)/H$. 
\end{proof}

We now show that these groups with either of these two actions are also EKR robust. 

\begin{thm}\label{thm:gendi}
Let $A$ be an abelian group. Both permutation actions of $D(A)$ have the EKR property and $D(A)$ is EKR robust. 
\end{thm} 

\begin{proof}
 From Corollary~\ref{regular}, the regular action is EKR robust, so we will only consider the degree $n$ action on $D(A)/H$ where $H =\langle x \rangle$. Note that given $a \in A$, we have $xaH=a^{-1}H$.

Observe that given $a\in A$, the stabilizer of $aH\in D(A)/H$ is $aHa^{-1}=\langle xa^{-2} \rangle$. Given $a,b \in A$, we have $aHa^{-1} = bHb^{-1}$ if and only if $(ab^{-1})^{2}=\id$. Therefore, $D(A)$ is Frobenius whenever $n$ is odd. Thus, in the case where $n$ is odd, Corollary~\ref{regular} again applies to $D(A)$.  We may therefore assume that $n$ is even.

Since $A$ acts regularly, it induces a clique of size $n$ in the graph $\Gamma_{D(A)}$, and thus the clique-coclique bound (Theorem~\ref{clique-coclique}) implies that an independent set cannot be larger than 2. As $H$ induces an independent set of size $2$, it follows that 
$\alpha( \Gamma_{D(A)} ) = 2$, and thus $D(A)$ satisfies the EKR property. 

As $A$ is not an elementary abelian $2$-group, it follows that there is a $y \in A$ such that $y^{2}\neq \id$. We have $D(A)_{H\to y^{2}H}=\{y^{2},\ xy^{-2}\}$. As $xy^{-2}$ fixes $yH$, it is not a derangement, and thus $D_{H\to y^{2}H}=\{y^{2}\}$. We therefore have $\de_{D(A)}\le1$. For any two distinct cosets $yH, zH$ in $D(A)/H$ we can choose the representatives so that $y, z \in A$, and therefore $zy^{-1} \in A$ maps $yH$ to $zH$. Remembering that $A \setminus \{\id\} \subseteq \Der(D(A))$, since $yH$ and $zH$ are distinct, $zy^{-1} \neq \id$, so $zy^{-1}$ is a derangement. Therefore $\de_{D(A)} \ge 1$, so $\de_{D(A)}=1$. From this, EKR robustness follows trivially.
\end{proof}

To explore the robustness of $D(A)$ in more detail, we compute the independence numbers of subgraphs of $\Gamma_{D(A)}$ obtained by removing exactly one label. We first mention an elementary result that will be useful.

\begin{lem}\label{lem:fixedpoints}
Let $A$ be an abelian group of order $n$. Consider $D(A) \le \sym{n}$ in its action on the cosets of $H=\langle x\rangle$. 
A non-identity element $z \in D(A)$ fixes a point if and only if $z=xa^{2}$, for some $a\in A$.

Accordingly, if $y,z \in xA$ are not derangements, then there is some $d \in A$ such that $z=yd^2$ and $y,z$ are conjugate in $G$.
\end{lem}
\begin{proof}
We have $z$ fixes $aH$ if and only if $z \in aHa^{-1}=\{\id,\ xa^{-2}\}$. This proves the first statement. 

If $y$ and $z$ are both a non-derangements then there is some $a,b \in A$ such that $z=xa^2$ and $y=xb^2$.
Then $y = xb^2 = z a^{-2} b^2 = z  (a^{-1}b)^2$, (since $A$ is abelian) completing the proof.
\end{proof}

\subsection{Removal a single label from connection set} 

In this section we consider in detail the different effects of removing a single label from the set of derangements of a generalized dihedral group.

\begin{prop}\label{remove-one}
Let $n$ be even, and $A$ an abelian group of order $n$. Let $ G=D(A) \le \sym{n}$, considered in its action on the cosets of $H=\langle x\rangle$.
If $D = \{d,d^{-1}\} \subseteq \Der(G)$, then $\alpha(\cay(G, \Der(G) \setminus  D)) \in \{2,4,6\}$.  Further, 
\[
\alpha(\cay(G, \Der(G) \setminus  D)=\begin{cases}
6, & \text{if and only if $d \in A$ and  $o(d)=3$;}\\
4, & \text{if and only if $o(d) \neq 3$ and $d =a^{2}$, for some $a\in A$;}\\
2, & \text{otherwise.}
\end{cases}
\]

\end{prop}

\begin{proof} 
First note that $\{ \id, d \}$ is an independent set, so $\alpha(\cay(G, \Der(G) \setminus  D)) \geq 2$. We now examine the elements of $xA$.

The action of $A$ is regular, so $A$ induces a clique of size $n$ in $\Gamma_G$. If $d \in xA$, then $A$ still induces a clique of order $n$ in $\cay(G, \Der(G) \setminus  D)$, and by the clique-coclique bound (Theorem~\ref{clique-coclique}) $\alpha(\cay(G, \Der(G) \setminus  D))\le 2$, so in this case $\alpha(\cay(G, \Der(G) \setminus  D)) = 2$. Henceforth we assume $d \in A$.

Let $S$ be a maximum independent set in $\cay(G, \Der(G) \setminus  D)$. By vertex-transitivity, we may assume $\id \in S$. Since $d$ and $d^{-1}$ are the only vertices in $A$ that are not adjacent to $\id$, we must have $S \cap A \subseteq \{\id, d, d^{-1}\}$. Likewise, $S \cap xA \subseteq \{z, zd, zd^{-1}\}$ for some $z \in xA$, so $S \subset \{\id, d, d^{-1}, z, zd, zd ^{-1}\}$, making $|S| \le 6$. Given $d \in \Der(A)$ and $z \in xA$, set $S_{d,z}:=\{\id, d, d^{-1}, z, zd, zd ^{-1}\}.$    We now prove the first case. 

Assume that $d \in A$ with $o(d)=3$, so $d=d^{-2}$. Therefore, by Lemma~\ref{lem:fixedpoints}, $\{ x, \, xd^{-1}=xd^{2},  \,xd=xd^{-2} \} \subseteq xA \setminus \Der(G)$. From this it is straightforward to verify that the set $S_{d,x}=\{\id, d, d^{-1}, x, xd, xd^{-1} \}$ is an independent set of size $6$ in $\cay(G, \Der(G) \setminus  D)$. 

Conversely, assume that there is an independent set $S$ of size $6$ in $\cay(G, \Der(G) \setminus  D)$. From above, we may assume without loss of generality that $S=S_{d,z}$, for some $z\in xA$. We must have $d\neq d^{-1}$ and thus $d^{2}\neq \id$. The edge between $d$ and $d^{-1}$ in $\Gamma_{G}$ is labeled by $\{d^{2},\ d^{-2}\}$, and by the independence of $S$ in $\cay(G, \Der(G) \setminus  D)$, it follows that $d^{2}=d^{-1}$ and thus $o(d)=3$.  Therefore, $\alpha(\cay(G, \Der(G) \setminus  D))=6$ if and only if $o(d)=3$. 

We now prove the second and third cases. Assume that $d\in A\setminus \{\id\}$ and $o(d)\neq 3$ and let $S$ be an independent set in $\cay(G,\ \Der(G)\setminus D)$. Without loss of generality, $S\subset S_{d,z}$, for some $z \in xA$. Also, either $d^{2}$ is a derangement or $d^2=\id$, and either of these implies $|S \cap \{d, d^{-1}\}| \leq 1$ and   
$| S \cap \{zd, zd^{-1}\} | \leq 1$. Therefore, we have 
\begin{equation}\label{eq:s4bound}
|S| = | S\cap \{\id, z\} | + | S\cap \{d, d^{-1}\} | + | S\cap \{zd, zd^{-1}\} | \leq 4,
\end{equation}
and thus $\alpha(\cay(G,\ \Der(G) \setminus D)) \leq 4$.

Suppose that $d=a^2$ for some $a \in A$. Then by Lemma~\ref{lem:fixedpoints}, $x$ and $xd$ are not derangements. It is now straightforward to verify that $\{\id,\ d,\ x,\ xd\}$ is an independent set of size $4$ and thus $\alpha(\cay(G,\ \Der(G)\setminus D))=4$.

To complete the proof of the second case, we prove a stronger form of the converse of the previous statement: namely, that if the independence number is not $2$ (or $6$), then there is some $a \in A$ such that $d=a^2$ (and therefore as we have just shown, the independence number is $4$). 

Suppose that  $\alpha(\cay(G,\Der(G)\setminus D))>2$, so there is an independent set $S$ of size $3$. Without loss of generality, we may still assume that $\id \in S$. 
We divide the remainder of the proof into two cases, depending on whether or not $z \in S$.
 
Case 1: Suppose that $z \in S$. As $|S|=3$, $\{\id, z\} \subset S$, $|S\cap \{d, d^{-1}\}| \leq 1$ and   
$|S\cap \{zd, zd^{-1}\}| \leq 1$, we  conclude that $S$ must be an element of 
\[ 
I:=\{\{\id,\ z,\ d\},\ \{\id,\ z,\ d^{-1}\},\ \{\id,\ z,\ zd\},\ \{\id,\ z,\ zd^{-1}\}\}.
\] 
If an element  of $I$ is an independent set in $\cay(G,\ \Der(G)\setminus D)$, then either (a) $z, zd \in xA \setminus \Der(G)$ (this arises from the first and third sets in $I$); or (b) $z,zd^{-1} \in xA \setminus \Der(G)$ (this arises from the second and fourth sets in $I$). Using Lemma~\ref{lem:fixedpoints}, it follows in either case that $d=a^{2}$, for some $a \in A$. 

Case 2: Assume $z \notin S$. Arguing as in the previous case, we see that $S$ must be an element of 
\[ 
J:=\{  \{\id,\ d^{-1},\ zd\},\ \{\id,\ d^{-1},\ zd^{-1}\},\  \{\id,\ d,\ zd\},\ \{\id,\ d,\ zd^{-1}\} \}.
\]
If an element of $J$ is an independent set, then one of the following is true: (c) $z,zd \in xA \setminus \Der(G)$ (this arises from the first set in $J$); or (d) $zd^{-1}, zd^{-2} \in xA \setminus \Der(G)$ (this arises from the second set in $J$; or (e) $zd,zd^{2} \in xA \setminus \Der(G)$ (this arises from the third set in $J$); or (f) $z,zd^{-1} \in xA \setminus \Der(G)$ (this arises from the fourth set in $J$).

If $zd^{2} \in xA \setminus \Der(G)$, then since $zd^2=d^{-1}zd$ is conjugate to $z$, it follows that $z \in xA \setminus \Der(G)$. Similarly, if $zd^{-2} \in xA \setminus \Der(G)$, it follows that $z \in xA \setminus \Der(G)$. Therefore (e) implies (c), and (d) implies (f).
Now either (c) or (f) must hold; that is, $z,zd \in xA\setminus \Der(G)$ or $z,zd^{-1} \in xA\setminus \Der(G)$. Thus by Lemma~\ref{lem:fixedpoints}, we must have $d=a^{2}$, for some $a \in A$. 

This concludes the proof.
\end{proof}


\subsection{Removal of two labels from Dihedral groups}

We now focus our attention to just the dihedral groups (that is, the case where $A = C_n$ is a cyclic group), and determine all of the possible independence numbers that can arise from removing two labels. We assume $n$ is even, as otherwise every element of $xC_n$ fixes a point, so $A=\Der(D(A))\cup \{\id\}$ and we have already considered this situation in some detail in Section~\ref{sect:cliques}. In this context, by a ``rotation" we mean any element of $C_n$, while a ``reflection" is any element of $xC_n$. 

Note that in a cyclic group $C_n$ of even order $n$, the squares are exactly the even permutations. 
If we apply Proposition~\ref{remove-one} with $d \in C_n$ and $D=\{d, d^{-1}\}$, the resulting independence number of $\cay(D(C_n),\Der(D(C_n))\setminus D)$ will be $2$ when $d$ is an 
odd permutation, while if $d$ is even, the independence number is $4$, unless $o(d)=3$, in which case it will be $6$.

We make a general observation that we will require in the proof of Proposition~\ref{prop:remove-two}. Derangements in $\sym{n}$ that are involutions have to be products of $n/2$ disjoint $2$-cycles. Since every element of $xA$ is an involution, we have the following result; the ``furthermore" is a straightforward observation about the cycle structures of involutions in dihedral groups of even degree.

\begin{lem}\label{lem:parityofderangements}
Let $n$ be even and $A$  an abelian group of order $n$, and consider the group $D(A)\le \sym{n}$.
The parity of a derangement in $xA$ is the same as that of $n/2$. Furthermore, if $A\cong C_n$ then every element of $xA$ that has the same parity as $n/2$ is a derangement.
\end{lem} 

Some parts of the proof of next result are easy consequences of more general results that we present in the following section, so we will  forward reference those rather than proving the results directly.

\begin{prop}\label{prop:remove-two}
Let $n$ be even, and $G=D(C_n)\le \sym{n}$, considered in its action on the cosets of $H=\langle x\rangle$. Let $\sigma$ be a generator for the index-2 subgroup $C_n$ in $G$. If $z,y \in \Der(G)$ with $z\neq y^{\pm1}$ and $D= \{z^{\pm1},y^{\pm1}\}$, then 
\[
\alpha(\cay(G,\Der(G)\setminus D))=
\begin{cases}
2, & \text{if $z$ and $y$ are both odd rotations, or both reflections;}\\
10, &\text{if $5 \mid n$ and }\{z^{\pm1},y^{\pm1}\}=\{\sigma^{n/5},\sigma^{2n/5},\sigma^{3n/5},\sigma^{4n/5}; \}\\
8, & \text{if $8 \mid n$ and }\{z^{\pm1},y^{\pm1}\}=\{\sigma^{n/4},\sigma^{n/2},\sigma^{3n/4};\}\\
6, & \text{if $z, y$ are even rotations, $z=y^2$ or $y=z^2$, }\\
 & \text{and none of the above holds;}\\
6, & \text{if either $o(z)=3$ or $o(y)=3$;}\\
4, & \text{otherwise}.
\end{cases}
\]
\end{prop}

\begin{proof}
If $z$ and $y$ are both reflections, or both odd rotations, then the result follows from Lemmas~\ref{remove-reflections} and~\ref{remove-odd-rotations}.

If $z$ and $y$ consist of one rotation and one reflection from $\Der(G)$, without loss of generality let $z$ be the rotation. First suppose that $z$ is an odd permutation. The subgraphs induced on $C_n \cap \alt{n}$ are cliques of order $n/2$, so the clique-coclique bound (Theorem~\ref{clique-coclique}) implies $\alpha(\cay(G,\Der(G)\setminus D))\le 4$. We claim that $\{\id, z,y,zy\}$ is an independent set of cardinality $4$. Certainly there are no edges between $\id$ and either $z$ or $y$, since $z$ and $y$ are in $D$; likewise there is no edge between $z$ and $zy$. Since $zy$ is a reflection with the opposite parity from $y$ and $y \in \Der(G)$, we have $zy \notin \Der(G)$, so is not adjacent to $\id$. This also shows there is no edge between $y$ and $z$, since $y^{-1}z=yz$ also has the opposite parity from $y$. Finally, $y^{-1}zy=z^{-1}$ which is also in $D$, so $y$ is not adjacent to $zy$.

Now suppose that $z$ is an even permutation. Then $\alpha(\cay(G,( \Der(G)\setminus D ) \cup\{y\})) \geq 4$, by Proposition~\ref{remove-one}, so $\alpha(G, \Der(G)\setminus D)\ge 4$, and if $o(z)=3$ then $\alpha(G, \Der(G)\setminus D)  \ge 6$. Since $\Der(G)\setminus  D$ contains every rotation except $z^{\pm1}$, an independent set can contain at most the three vertices $t$, $tz$, and $tz^{-1}$ from $C_n$, for some choice of $t$; similarly, it can contain at most three vertices from $xC_n$. This shows that if $o(z)=3$ then $\alpha(\cay(G,\Der(G)\setminus  D))=6$. If $o(z) \neq 3$ then for any such choice of $t$, the vertices $tz$ and $tz^{-1}$ are equal or adjacent, so the independent set can contain at most two vertices from each of $C_n$ and $xC_n$, showing that $\alpha(\cay(G, \Der(G)\setminus  D))=4$.

The only remaining possibility is that $z$ and $y$ are both rotations, and at least one of them, say $z$, is even. Note that in this case, the only vertices in $C_n$ that are not adjacent to some fixed vertex $t \in C_n$, are $tz$, $tz^{-1}$, $ty$, and $ty^{-1}$, and similarly in  $x C_n$. Thus $\alpha(\cay(G,\Der(G)\setminus  D)) \le 10$. Furthermore, by Proposition~\ref{remove-one}, since $z$ is even, $\alpha(\cay(G, \Der(G)\setminus  D))\ge 4$. 

Now let us suppose that $y$ is odd. 
Let $\tau \in \Der(G)$ be a reflection, then by Lemma~\ref{lem:parityofderangements}, $\Der(G)\setminus  D$ contains every reflection that has the same parity as $\tau$. This implies that every vertex in $xC_n$ is adjacent to either $t$ or $ty$ for any $t \in C_n$. So any independent set $S$ must either be a subset of $C_n$, or a subset of $t(C_n \cap \alt{n}) \cup ty\tau(C_n \cap \alt{n})$. Furthermore, since $y$ is odd, $y^2 \in \Der(C_n) \setminus D$ so $ty$ and $ty^{-1}$ cannot both be in $S$. Likewise, if $o(z) \neq 3$ then $tz$ and $tz^{-1}$ cannot both be in $S$. Thus, if $S \subseteq C_n$ then $|S|\le 4$ unless $o(z)=3$ in which case $|S|\le 5$. On the other hand, if $S \subseteq t(C_n \cap \alt{n}) \cup ty\tau(C_n\cap \alt{n})$, then $S$ contains at most $2$ elements of each of the two sets (some $t'$ together with either $t'z$ or $t'z^{-1}$) unless $o(z)=3$ in which case it contains at most $3$ elements of each of the two sets. Putting all of this together, we see that a maximum independent set has $4$ elements unless $o(z)=3$, in which case it has $6$ elements.

We may henceforth assume that both $x$ and $y$ are even rotations. 
Since $\Der(G)\setminus  D$ contains all odd rotations and all reflections in $\tau(C_n\cap \alt{n})$ (see Lemma~\ref{lem:parityofderangements}), if any vertex in an independent set lies in $C_n\cap \alt{n}$, then there is no vertex in the independent set that lies in either $C_n \setminus \alt{n}$ or $\tau(C_n\cap \alt{n})$. Using vertex-transitivity, we may assume that an independent set of maximum cardinality is some subset of $\{\id, z^{\pm1},y^{\pm1},t,tz^{\pm1},ty^{\pm1}\}$ for some non-derangement reflection $t$.

We are aiming to determine precisely when we can have an independent set $S$ with $|S|>4$. If $|S|>4$ then without loss of generality using vertex-transitivity, we may assume $|S\cap (C_n \cap \alt{n})|\ge 3$, and $\id \in S$. This leads us to one of two situations: either (after replacing $z$ and/or $y$ by its inverse if necessary), $\{\id, z, y \} \subseteq S$, or $S \cap (C_n \cap \alt{n})=\{\id, z^{\pm1}\}$, and left-translates of this are the only independent sets of cardinality $3$ in either $C_n$ or $xC_n$. We consider the latter possibility first. 
 
 If $S \cap (C_n \cap \alt{n})=\{\id, z^{\pm1}\}$, and  left-translates of this are the only independent sets of cardinality $3$ in either $C_n$ or $xC_n$, then $\alpha(\cay(G),\Der(G)\setminus D))\le 6$. Note that since we assumed $|S|>4$, we must have $z \neq z^{-1}$ in this case. Furthermore, since $z$ and $z^{-1}$ are nonadjacent and $z \neq z^{-1}$ we must have $z^2 \in D$. If $z^2=y$ or $z^2=y^{-1}$ then $\{\id, z, z^2\}$ is an independent set in $C_n$ that is not a left-translate of $S$, a contradiction. Since we cannot have $z^2=z$, the only possibility is $z^2=z^{-1}$, meaning $o(z)=3$. Since we already have $\alpha(\cay(G,\Der(G)\setminus D))\le 6$ and we know there is an independent set of cardinality $6$ when $o(z)=3$ (even if only the label $\{z,z^{-1}\}$ has been removed), we obtain $\alpha(\cay(G,\Der(G)\setminus D))= 6$ in this case.

We may now assume without loss of generality that $\{\id, z, y\} \subseteq S$. This implies $z^{-1}y \in D$, so we must have $z^{-1}y \in \{\id, z^{\pm1},y^{\pm1}\}$. Since $\id$, $z$, and $y$ are all distinct and $z \neq y^{\pm 1}$, the only possibilities are $z^{-1}y=z$ or $z^{-1}y=y^{-1}$. So $y=z^2$ (or $z=y^2$, but in this case we can exchange $y$ and $z$). 

Now we have $\{\id,z,y\}= \{\id, z,z^2\}$  independent; taking the union with some $\{t, tz,tz^2\}$ (which is also an independent set) shows $\alpha(\cay(G, \Der(G)\setminus  D))\ge 6$. We must still determine when it can happen that $\alpha(\cay(G, \Der(G)\setminus  D))>6$. In this case, a maximum independent set must contain both $z$ and $z^{-2}$, or 
both $z^{-1}$ and $z^{2}$, in either case this requires that $o(z)>3$. 

If $o(z) \ge 6$ then $z^3 \notin \{z^{\pm1},z^{\pm2}\}$ so there is an edge between $z^{-1}$ and $z^2$, and between $z^{-2}$ and $z$, so as observed in the preceding paragraph, we cannot have an independent set of size greater than $6$. The only possibilities that remain that could lead to higher independence numbers are $o(z) \in \{4,5\}$. It is not hard to verify that in each of these cases, the vertices $\{\id, z^{\pm1}, z^{\pm2}, t, tz^{\pm1}, tz^{\pm2}\}$ are all independent; however, if $o(z)=4$ then there are only $8$ distinct vertices in this set. 

Noting that $o(z)=5$ requires $5 \mid n$ since $z \in C_n$, and that $o(z)=4$ and $z$ being even  requires $z$ to consist of an even number of $4$-cycles, so $8\mid n$.
\end{proof}

\subsection{Robustness in generalised dihedral groups}

In this subsection we obtain some robustness results similar to Proposition~\ref{robust} for generalised dihedral groups. In some of these cases the independence number remains at $\alpha(\Gamma_G)$, rather than jumping up to $2\alpha(\Gamma_G)$. In this subsection we consider the action of $D(A)$, where $A$ is an abelian group, on $D(A)/H$, where $H=\langle x\rangle$. Associated with this action is a natural map $\sigma: D(A) \to \sym{D(A)/H}$. 
For ease of notation, we identify $\sym{D(A)/H}$ with $\sym{n}$ and identify any $g \in D(A)$ with $\sigma(g)$. We further assume that $n$ is even (recall that if $n$ is odd, then since every element of $xA$ has order $2$, $\Der(D(A))\cup \{\id\}=A$, which was dealt with in Section~\ref{sect:cliques}).

The next two results show that many elements can be removed from the connection set of the derangement graph without increasing the independence number. In Lemma~\ref{remove-reflections}, many elements from $xA$ will be removed without changing the independence number of the graph, and a similar result is shown in Lemma~\ref{remove-odd-rotations}, but with elements from $A$ being removed. For both results we describe the Cayley graph with the elements remaining in the connection set, rather than describing the derangements that are removed.

\begin{lem}\label{remove-reflections}
Let $n$ be even and $A$ an abelian group of order $n$. Consider the group $G=D(A) \le \sym{n}$. 
If $A \setminus\{\id\} \subseteq C \subseteq \Der(G)$, then $\alpha(\cay(G,C))=\alpha(\Gamma_G)=2$.
\end{lem}

\begin{proof}
The graph $\cay(G,C)$ still has the cliques of order $n$ induced by $A$ and $xA$, so by the clique-coclique bound, a maximum independent set has cardinality at most $2$. It cannot be smaller since $\Gamma_G$ is not complete.
\end{proof}

So in the above situation, although it is possible to double the independence number by removing a single element of $A$, it is also possible to remove many labels without increasing the independence number at all. Exactly how many labels are actually being removed from the total number of derangement labels, depends on how many elements of $A$ are involutions, and on how many elements of $xA$ are derangements. In the case of $D(C_n)$ where $n$ is even, there are $n/2$ labels that we can remove (half of the reflections are derangements and each derangement is a label), and $n/2$ labels remain in the connection set if we remove all of them (there are $n$ elements of $C_n$, one of which is not a derangement, and one of which is an involution and therefore a label by itself).

\begin{lem}\label{remove-odd-rotations}
Let $n$ be even and $A$ an abelian group of order $n$, with $A  \nleq \alt{n}$, consider the group $G=D(A) \le \sym{n}$. Let $\tau$ be any element of $ xA \cap \Der(G)$. If 
\[
( (A\cap \alt{n}) \cup \tau(A\cap \alt{n}) ) \setminus\{\id\} \subseteq C \subset \Der(G),
\] 
then  $\alpha(\cay(G,C))=\alpha(\Gamma_G)=2$.
\end{lem}

\begin{proof}
First, we know that $\alpha(\Gamma_G) = 2$, so $\alpha(\cay(G,C)) \geq 2$. 

Since $A$ is not a subgroup of $\alt{n}$, it has $n/2$ even permutations, and $n/2$ odd permutations, and therefore so does $xA$. Since $\tau$ is a derangement, it has the same parity as $n/2$ by Lemma~\ref{lem:parityofderangements}. Then for any $a_1, a_2 \in A$ that are both even or both odd, we have $a_1$ is adjacent to $a_2 \tau$, since  $a_1^{-1} a_2 \tau \in \tau(A\cap \alt{n}) \subset \Der(G)$ by hypothesis. Thus the subgroup formed by $ (A\cap \alt{n})\cup \tau(A\cap \alt{n})$ induces a clique of order $n$ in the graph $\cay(G,C)$, so by the clique-coclique bound (Theorem~\ref{clique-coclique}), a maximum independent set has cardinality $2$. 
\end{proof}

Next we will show that if $A \nleq \alt{n}$, then the independence number will always at least double when it increases at all; this can be thought of as another version of robustness, which we also saw in Section~\ref{sect:cliques}, and will contrast with the case when $A \le \alt{n}$, considered in the following section. 

\begin{cor}
Let $n$ be even, $A$ an abelian group of order $n$ with $A \nleq \alt{n}$, and $G=D(A)\le \sym{n}$.
There is no inverse-closed $C \subseteq \Der(G)$ such that 
\[
\alpha(\Gamma_G) <\alpha(\cay(G,C))<2\alpha(\Gamma_G).
\]
\end{cor}

\begin{proof}
We know that $\alpha(\Gamma_G)=2$, so our goal, then, is to show that we can never have $\alpha(\cay(G,C))=3$.

Since $A \nleq \alt{n}$, we have $|A\cap \alt{n}|=n/2=|A \setminus \alt{n}|$. Thus we must also have $|xA \cap \alt{n}|=n/2=|xA \setminus \alt{n}|$. Since all derangements in $xA$ have the same parity (namely the parity of $n/2$, see Lemma~\ref{lem:parityofderangements}), there is some $z \in xA$ such that $z(A \cap \alt{n}) \cap \Der(G)=\emptyset$; note this implies $xA \cap \Der(G) \subseteq z(A\setminus \alt{n})$.

Suppose first that there is some nonidentity element $d \in A \cap \alt{n}$ such that $d \notin C$. Then $\{\id, d, z, zd\}$ is an independent set since $z, zd, d^{-1}zd \in z(A \cap \alt{n})$ so are not derangements and not in $C$. Thus $\alpha(\cay(G,C)) \ge 4$.

Now suppose that there are some $d_1,d_2 \in A \setminus \alt{n}$  such that $d_1, zd_2 \notin C$. Recall that every element of $xA \cap \Der(G)$ lies in $z(A\setminus \alt{n})$ by our choice of $z$. Then $\{\id, d_1,zd_2,d_1zd_2\}$ is an independent set. This is because $d_1 \notin C$ implies $d_1^{-1} \notin C$ and as $zd_2 \notin C$ also, there are no edges between $\id$ and $d_1$ or $zd_2$, or between $d_1zd_2=zd_2d_1^{-1}$ and $zd_2$ or $d_1$. Also, $d_1^{-1}zd_2=zd_2d_1$ and $d_1zd_2$ are both in $z(A \cap \alt{n})$ so are not in $C$. Thus $\alpha(\cay(G,C)) \ge 4$ in this case also.

The only remaining possibilities are that either $A \setminus \{\id\} \subseteq C$, or $(A \cap \alt{n})\setminus \{\id\} \subseteq C$ and $z(A\setminus \alt{n}) \subseteq C$. In the first of these cases, Lemma~\ref{remove-reflections} tells us that $\alpha(\cay(G,C))=2$. In the second case, Lemma~\ref{remove-odd-rotations} tells us that $\alpha(\cay(G,C))=2$.
\end{proof}


\subsection{Generalised dihedral groups over subgroups of the alternating group}

The results we have proved so far about generalised dihedral groups do not necessarily imply that it is not possible to remove a subset of labels from the connection set and produce a graph whose independence number is $3$ (i.e.~bigger than a star but smaller than a binary star). In fact, this can happen sometimes; we will show exactly when it can happen in this section. In particular, we will show that if $n$ is even and $A \le \alt{n}$, then we can always find a connection set $D$ such that $\alpha(\cay(D(A),D))=3$.

Recall from Lemma~\ref{lem:fixedpoints} that any two non-derangements in $D(A)$ must differ by an element of $A$ that is a square.
So next we determine how many of the elements of $A$ are squares.

\begin{lem}\label{quarter-square}
Let $n$ be even and let $A$ be a transitive abelian group of order $n$. If $A \le \alt{n}$ then at most $|A|/4$ elements of $A$ are squares.
\end{lem}

\begin{proof}
Suppose $A_1=\langle a\rangle$ is a cyclic subgroup of $A$ that has even order $k$. Since $A$ is abelian and transitive, it is regular, so every cycle in $a$ has even length $k$. Since $A \le \alt{n}$, the number of cycles in $a$, which is $n/k$, must also be even. Thus $|A:A_1|$ is even. By choosing $A_1$ maximal (so that it is a direct factor of $A$), this implies that it must be possible to write $A \cong A_1 \times A_2$ where $|A_1|$ and $|A_2|$ are both even.

Since $A$ is abelian, $B=\{a \in A: o(a) \le 2\}$ is a subgroup of $A$. These are exactly the elements of $A$ that square to $\id$, and any coset $cB$ consists of exactly the elements of $A$ that square to $c^2$. 
So the number of squares in $A$ is exactly the index of $B$ in $A$. Thus, what we must show is that if $A \le \alt{n}$ then $|B| \ge 4$.

Since $A \cong A_1 \times A_2$ and $|A_1|$, $|A_2|$ are even, there exist involutions $a_1 \in A_1$, $a_2 \in A_2$, and each of $a_1$, $a_2$, $a_1a_2$ and $\id$ have order at most $2$ in $A$, so $|B| \ge 4$. Thus the number of squares in $A$ is at most $|A|/4$.
\end{proof}

We can now bound the number of non-derangements in $D(A)$.

\begin{cor}\label{quarter-der}
Let $n$ be even and $A$ a transitive abelian group of order $n$. If $A \le \alt{n}$, then at most $|A|/4$ elements of $xA$ in $D(A)\le \sym{n}$ are not derangements.
\end{cor}

\begin{proof}
Let $G=D(A)$, and let $y \in xA$ such that $y \notin \Der(G)$. By Lemma~\ref{lem:fixedpoints}, for every $z \in xA$ with $z \notin \Der(G)$, there is some $a \in A$ such that $z=ya^2$. By Lemma~\ref{quarter-square}, at most $|A|/4$ distinct elements of $A$ can be written as $a^2$ for some $a \in A$, including $\id$. Thus the number of distinct choices for $z$ (including $z=y$) is at most $|A|/4$.
\end{proof}

The above results allow us to show that whenever $A$ was not covered by the results in the previous subsection, we can in fact find a set of labels to remove that allows us to increase the cardinality of a maximum independent set by just one, rather than doubling it. 


\begin{prop}
Let $n$ be even and $A \le \alt{n}$ a transitive abelian group that is not elementary abelian. Let $G=D(A) \le \sym{n}$. Then $\alpha(\Gamma_G)=2$ and there is some $D \subset \Der(G)$ such that $\alpha(\cay(G, \Der(G) \setminus D))=3$.
\end{prop}

\begin{proof}
Let $a$ be an element of maximum order in $A$; since $|A| = n$ is even we must have $o(a)$ even 
and since $A$ is not elementary abelian $o(a)>2$. Since the order of $a$ is the maximum element order, $a$ cannot be a square, so by Lemma~\ref{lem:fixedpoints}, if $y \in xA \setminus \Der(G)$ then $ya \in xA\cap \Der(G)$.

By Corollary~\ref{quarter-der}, at most one-quarter of the $|A|$ elements of $xA$ in $D(A)$ are not derangements; putting this together with our conclusion of the previous sentence, there must be some $y \in xA$ such that $y, ya \in \Der(G)$.

Let $D= \{a^{\pm1},y,ya\}$, and consider $\cay(G, \Der(G) \setminus D)$. It is clear that $\{\id, y,ya\}$ is an independent set (based on the elements in $D$). We claim that there is no independent set of cardinality greater than $3$, so that $\alpha(\cay(G, \Der(G) \setminus D))=3$.

By vertex-transitivity, we may assume that a maximum independent set $S$ in $\cay(G,  \Der(G) \setminus D)$ contains $\id$. Since $a^{\pm1}$ are the only vertices in $A$ that are not adjacent to $\id$, we have $S\cap A \subseteq \{\id, a^{\pm1}\}$. Since $o(a)>2$ is even, $a^2 \notin \{\id, a^{\pm1}\}$, so $a^2 \not\in  D$.  
Thus $a$ and $a^{-1}$ cannot both be in $S$. Without loss of generality (interchanging $a$ with $a^{-1}$ if necessary) we may assume that $S \cap A\subseteq \{\id, a\}$. 

Similarly, there must be some $z \in xA$ such that $S \cap xA \subseteq \{z, za\}$. So $S \subseteq \{\id,a,z,za\}$ and we already see that $\alpha(\cay(G,  \Der(G) \setminus D)) \le 4$. It remains to show that $\{\id,a,z,za\}$ is not an independent set for any choice of $z \in xA$.

Recall that $a$ is not a square, so by Lemma~\ref{lem:fixedpoints}, at least one of $z, za$ must be a derangement. In order for $\{\id,a,z,za\}$ to be independent, any derangement in $\{z,za\}$ must be in $D$, so must be either $y$ or $ya$. Thus $\{z,za\} \cap \{y,ya\} \neq \emptyset$.

Suppose that either $z=y$ or $za=ya$ (either of these implies the other). Consider the adjacency of $a$ and $ya$, which depends on whether or not $a^{-1}ya$ is in $D$. Since $o(a)>2$ we cannot have $a^{-1}ya=y $ (as $yay=a^{-1}$), or $a^{-1}ya=ya$, so $a^{-1}ya$ is not in $D$. But also, since $a^{-1}ya$ is conjugate to $y$ and $y \in \Der(G)$ we also have $a^{-1}ya \in \Der(G)$. Therefore $a^{-1}ya \in  \Der(G) \setminus D$ and the set $\{\id,a,y,ya\}$ is not independent.

Suppose that $z=ya$, so $za=ya^2=a^{-1}ya$. Just as in the previous paragraph, $a^{-1}ya \in  \Der(G) \setminus  D$, so $za \in \Der(G) \setminus  D$, and again $\{\id,a,z,za\}$ is not independent.

Finally, suppose that $za=y$, so $z=ya^{-1}=ay=a(ya)a^{-1}$. Using a similar argument to the previous ones, it is straightforward to conclude that $z \in \Der(G) \setminus  D$ and therefore $\{\id,a,z,za\}$ is not independent.
\end{proof}


\section{$\pgl(2,q)$ and symmetric groups}\label{sect:other groups}

In this section we consider the group $\pgl(2,q)$, in its natural $2$-transitive action, and the group $\sym{n}$ acting on points $\{1,\dots,n\}$. We are able to show that for both of these families of group many labels must be removed from the derangement graph before the resulting Cayley subgraph has a larger independent set.

\subsection{$\pgl(2,q)$}

Let $G=\pgl(2,q)$ in its natural action as a $2$-transitive permutation group of degree $q+1$.   
Pick an element in $G$ of order $q+1$ and consider the cyclic group $C_{q+1}$ it generates. This group $C$ forms a maximum clique in $\Gamma_{G}$.  So by the clique-coclique bound (Theorem~\ref{clique-coclique}), $\alpha(\Gamma_G) \leq |G|/|C_{q+1}| = (q-1)q$. Since this is exactly the size of the stabilizer of a point, $\alpha(\Gamma_G) = \frac{|G|}{q+1}$. It was shown in~\cite{MR2739502} that the number of derangements in $G$ is $\frac{q^2(q-1)}{2}$ (this is the degree of $\Gamma_{\pgl(2,q)}$). Since $\pgl(2,q)$ is 2-transitive, this implies that $|D_{i\to j}|=\frac{(q-1)q}{2}$; further, considering the order of the cycle structure of the derangements in $G$, this is also the value of $\de_G$. 

\begin{lem}
There are $(q-1)q/2$ distinct conjugates of $C_{q+1}$ in $G$, and
    $C' \cap C=\id $, for any distinct conjugates $C$ and $C'$ of $C_{q+1}$.
\end{lem}
\begin{proof}
    To prove this, we use the subgroup structure of $\pgl(2,q)$. This is well-known and was first described by Dixon~\cite{MR0104735}. For a modern description of this result, we refer to~\cite{MR2240756}. Theorem 2 of~\cite{MR2240756} lists all the subgroups of $\pgl(2,q)$; the second statement of this theorem includes the fact that there is one conjugacy class of subgroups that contains $(q-1)q/2$ subgroups isomorphic to the cyclic group $C_d$ whenever $d | q+1$. In particular, there are $(q-1)q/2$ subgroups conjugate to $C_{q+1}$. 

Corollary~4 of \cite{H1970} states that $C' \cap C =\id$, for any distinct conjugates $C$ and $C'$ of $C_{q+1}$.
\end{proof}

As we remove labels from $\Gamma_{\pgl(2,q)}$ the independence number will remain at $(q-1)q$ as long as there is at least one clique of size $q+1$ in the connection set. This leads to the following result which shows that $\pgl(2,q)$ is EKR robust.

\begin{thm}\label{thm:pgl2q}
Let $G=\pgl(2,q)$ in its natural action as a $2$-transitive permutation group of degree $q+1$, and let $D \subset \Der(G)$. If $D$ has fewer than $\de_G=(q-1)q/2$ labels, then $\alpha(\Gamma_{G}) =\alpha(\cay(G,  \Der(G) \setminus D) )$. 
Thus $G$ is EKR robust.
\end{thm}
\begin{proof}
Consider the $(q-1)q$ subgroups conjugate to $C_{q+1}$. If any one of these subgroups is contained in $\Der(G) \setminus D$, then the subgroup forms a clique of size $q+1$ in $\cay(G,  \Der(G) \setminus D)$. Then by the clique-coclique bound $\alpha(\cay(G,  \Der(G) \setminus D) ) = \frac{|G|}{q+1} = \alpha(\Gamma_{G})$.

So if $\alpha(\cay(G,   \Der(G) \setminus D) ) > \frac{|G|}{q+1}$, then there must be at least one element from each subgroup conjugate to $C_{q+1}$ in $D$. Since each of these subgroups intersect only at the identity, there must be at least $(q-1)q/2$ distinct  elements in $D$. Further, for any element in a subgroup, its inverse is clearly also in the subgroup. This means that there must be at least $\de_G=(q-1)q/2$ labels in $D$. Thus if $D$ has fewer than $\de_G$ labels, then $\alpha(\cay(G,   \Der(G) \setminus D) ) = \alpha(\Gamma_{G}) $. 
\end{proof}
    
    
\subsection{Symmetric groups} 

In this section we consider the symmetric groups $\sym{n}$ in their natural action on $[n]$.
We use $\Der(n)$ to denote the set of all derangements in $\sym{n}$ and $d_n  = |\Der(n)|$ (it is known that $d_n \sim n!/e$). The size of a star in $\sym{n}$ is $(n-1)!$, and these are the largest independent sets in $\Gamma_{\sym{n}}$. This can be found in~\cite{CameronKu,DF1977} but also follows easily from the clique-coclique bound (Theorem~\ref{clique-coclique}) since the derangement graph $\Gamma_{\sym{n}}$ has many cliques of size $n$. 

Indeed, any Latin square of order $n$ corresponds to a clique of size $n$ in $\Gamma_{\sym{n}}$. To see this correspondence, let $L$ be a Latin square of order $n$. For each row $i$ of $L$, define a permutation $\sigma_i$ that maps $j \in \{1, \dots, n\}$ to the entry in row $i$ at column $j$. Then the set of permutations $\sigma_1, \sigma_2, \dots, \sigma_n$ will form a clique in $\Gamma_{\sym{n}}$ since there are no repeated entries in a column (so no $\sigma_{i_1}(j) = \sigma_{i_2}(j)$). Further, any Latin square can be transformed so that the first row corresponds to the identity in $\sym{n}$, and then all the other rows correspond to derangements. 

The clique-coclique bound (Theorem~\ref{clique-coclique}) implies that if $\alpha(\cay(\sym{n}, \Der(n)  \setminus D)) > (n-1)!$, then $\cay(\sym{n}, \Der(n)  \setminus D)$ cannot have a clique of size $n$. This means for every Latin square of order $n$ at least one element from the set $\sigma_{r_1}^{-1}\sigma_{r_2}$ must lie in $D$. This clearly happens if $D_{i \to j} \subseteq D$. 

There is a set $D$, strictly smaller than $D_{i \to j} \cup D_{i \to j}$, for which $\cay(\sym{n}, \Der(n)  \setminus D )$ has no clique of size $n$. To define $D$, first consider the set $C$ of all derangements that satisfy the following three conditions: map 1 to 2; include the cycle (3,4); and do not include $(1,2)$. Set 
$D = ( D_{1 \to 2}  \cup D_{2 \to 1}) \backslash (C \cup C^{-1} )$. Assume $\cay(\sym{n}, \Der(n)  \setminus D )$ has a clique of size $n$, and we can assume this clique contains the identity. There must be a permutation in the clique that maps 1 to 2 and a permutation that maps 2 to 1. These permutations must both be in $C\cup C^{-1}$ (as we have removed every other permutation that maps 1 to 2, or 2 to 1).  Since $C\cup C^{-1}$ has no permutations with $(1,2)$, these two permutations must be distinct. But these permutations both map 3 to 4, so cannot be in a clique together. 
This example does not demonstrate that the symmetric groups are not robust, but it does show that we cannot use the clique-coclique bound in this case. We do not know the value of $\alpha( \cay(\sym{n}, \Der(n)  \setminus D ) )$ for all $n$; in Lemma~\ref{lem:sym5} we see that this graph does have a larger independent set when $n=5$, but we suspect this is not the case for larger $n$.

By the $2$-transitivity of $\sym{n}$, $|D_{i\to j}|=d_n/(n-1)$ for every distinct $i, j \in [n]$. 
We can give a formula for the number of distinct labels in $|D_{i\to j}|$. In this formula we need the number of derangements that are involutions, which is exactly $(n-1)!! = (n-1)(n-3) \cdots 1$ for $n$ even, and (with a slight abuse of notation) $(n-1)!! = (n-1)(n-3) \cdots 0 =0$ for $n$ odd (this is exactly the number of perfect matchings of $[n]$). 

\begin{lem}\label{lem:symnlabels}
The number of distinct labels in $D_{i\to j}$ for $G = \sym{n}$ is 
\[
\de_{\sym{n}}= \frac{d_n}{n-1} - \frac{ d_{n-2} - (n-1)!! }{2}.
\]
\end{lem}
\begin{proof}
Since $\sym{n}$ is 2-transitive, the number of elements in $D_{i\to j}$ is $d_n/(n-1)$.
Any element $\sigma$ with $\sigma$ and $\sigma^{-1}$ both in $D_{i\to j}$ includes the 2-cycle $(i,j)$. There are exactly $d_{n-2}$ such permutations, and $(n-1)!!$ of them are involutions. So there are $\frac12 \left( d_{n-2} - (n-1)!! \right)$ labels that are counted twice in $D_{i\to j}$.
\end{proof}

As proved in Proposition~\ref{prop:binarystars}, the graph $\cay(\sym{n}, \Der(n) \setminus ( D_{i\to j} \cup D_{j \to i} ) )$ will have an independent set of size $2 (n-1)!$ formed by a binary star $\sym{n}_{j \to j} \cup \sym{n}_{i \to j}$. 

In the special case where $n=p$ is a prime number, each $p$-cycle generates a subgroup that is equivalent to a Latin square and any two such subgroups intersect only at the identity. This can be used to get a lower bound on the number of labels needed in a set $D$ to ensure $\alpha(  \cay(\sym{p}, \Der(p) \setminus D) ) > (p-1)!$ when $p\ge 3$.

\begin{lem}
Let $p\ge 3$ be prime and $D \subset \Der(p)$.
If $\alpha( \cay(\sym{p},  \Der(p) \setminus D) ) > (p-1)!$, then $D$ must contain at least $(p-2)!$ labels. 
\end{lem}
\begin{proof}
Since $p$ is prime, each of the $(p-1)!$ distinct $p$-cycles generates a subgroup of order $p$ that has $p-1$ elements of order $p$. These produce $(p-2)!$ cliques of order $p$ in $\cay(\sym{p},  \Der(p) \setminus D)$ (one for each subgroup), that intersect only in the identity. Each of these cliques corresponds to a subgroup, so all the labels within the clique belong to the subgroup.  If $\alpha ( \cay(\sym{p},  \Der(p) \setminus D) ) > (p-1)!$, then by the clique-coclique bound (Theorem~\ref{clique-coclique}), $D$ must include at least one label from each of these $(p-2)!$ subgroups.
\end{proof}

In this case the number of labels that must be removed in order to produce a binary star is
\[
\de_{\sym{p}}= \frac{ d_{p}  }{p-1} - \frac{d_{p-2}}{2}
\sim \frac{p!}{e (p-1)} - \frac{(p-2)!}{2e}  
= (p-2)!\frac{2p-1}{2e}.
\]
So when $p>3$, this bound of $(p-2)!$ is lower than $\de_{\sym{p}}$.

 However, even after removing one label from each subgroup generated by a $p$-cycle, the resulting Cayley graph still has many cliques, since there are many Latin squares that have a different form. 

For small values of $n$ we can check using a computer search if $\sym{n}$ is robust.
For $n=3$, $\de_{\sym{3}} = 1$, so it is trivially robust. For $n=4$, a computer search reveals that if $\alpha( \cay(\sym{n},  D_n \setminus D) ) > (n-1)!$, then $D$ has 3 labels which is the number of labels in $D_{i\to j} \cup D_{j \to i}$. For $n=5$, there are some exceptions, that don't generalize easily to larger $n$. 

\begin{lem}\label{lem:sym5}
In $G=\sym{5}$ in its natural action on $[5]$, there exist sets $D$ with fewer labels than $D_{i \to j}$
and $\alpha ( \cay( G, \Der(5) \setminus \{D \cup D^{-1} \}) ) > (5-1)!$ . 
\end{lem}
\begin{proof}
In this case, there 10 labels in $D_{i \to j}$. Let 
\begin{align*}
D\cup D^{-1}=\{ & (1,3,2)(4,5), (1,2,3)(4,5), (1,3)(2,5,4), (1,3)(2,4,5) , (1,5,4)(2,3), \\
       &  (1,4,5)(2,3),  (1,3,5,4,2),  (1,2,4,5,3), (1,2,5,4,3), (1,3,4,5,2),   \\
       &  (1,2,3,5,4), (1,4,5,3,2),  (1,2,3,4,5),  (1,5,4,3,2),   \\
       & (1,3,2,4,5), (1,5,4,2,3), (1,3,2,5,4),(1,4,5,2,3) \}.
\end{align*}

(Note that $D \cup D^{-1}$ is the  set of all derangements in $D_{4 \to 5} \cup D_{5 \to 4}$, except $(1,2)(3,4,5)$ and $(1,2)(3,5,4)$.)
Then $D$ has 9 labels and the set $A$ below is an independent set in $ \cay(\sym{5}, \Der(5) \setminus \{D \cup D^{-1} \}) $ of size $30$:
\begin{align*}
A=\{& \id, (4,5), (3,4), (3,5,4), (3,4,5), (3,5), (2,3), (2,3)(4,5), \\
& (2,3,4), (2,3,5,4), (2,3,4,5), (2,3,5), (1,3,2), (1,3,2)(4,5), \\
& (1,3,4,2), (1,3,5,4,2), (1,3,4,5,2), (1,3,5,2), (1,2,3), \\
& (1,2,3)(4,5), (1,3), (1,3)(4,5), (1,2,3,4), (1,2,3,5,4), (1,3,4),\\
& (1,3,5,4), (1,2,3,4,5), (1,2,3,5),  (1,3,4,5), (1,3,5) \}.
\end{align*} 
 
Since an independent set in $\cay(\sym{5}, \Der(5) )$ is a star with size $4!=24$, the group $\sym{5}$ is not EKR robust.
\end{proof}

There remains much to be determined about how removing labels impacts the independence number of $\Gamma_{\sym{n}}$. We revisit this topic in Section~\ref{sect:open} with open problems and a conjecture.


\section{Subgroups and Homomorphisms}
\label{sect:subgroups} 

A \emph{homomorphism} from a graph $X$ into a graph $Y$ is a map from $V(X)$ to
$V(Y)$ that maps edges in $X$ to edges in $Y$. We write $X\to Y$
to denote that there is a homomorphism from $X$ to $Y$. The following lemma,
which is equivalent to the ``no-homomorphism lemma'' of Albertson and Collins~\cite{MR791653}, bounds the size of an independent set in a graph using a homomorphism.

\begin{lem}
    \label{lem:nohom}
    If $Y$ is a vertex-transitive graph and there is a homomorphism from $X$ to $Y$, then
	\[
		\frac{|V(X)|}{\alpha(X)} \leq \frac{|V(Y)|}{\alpha(Y)}.
	\]
	If $X$ is a vertex-transitive graph and equality holds, then the preimage in $X$ of an independent set
	of maximum size in $Y$ is an independent set of maximum size in $X$.\qed
\end{lem}

The situation we will consider here are Cayley graphs (specifically, the derangement graphs) on groups $H$ and $G$, with $H \leq G$. Then embedding is a graph homomorphism 
$\Gamma_H \to \Gamma_G$. The following is well-known, see for example~\cite[Theorem 14.6.2]{GMbook}.

\begin{thm}\label{thm:subtrans}
  Let $G$ be a transitive subgroup of $\sym{n}$ and let $H\le G$ also be transitive. If $H$ has the EKR property, then $G$ also has the EKR property. 
\end{thm}
\proof 
Since $H$ has the EKR property and is transitive, the size of
the maximum independent set is $|H|/n$. By Lemma~\ref{lem:nohom}
	\[
		\frac{|H |}{ \frac{|H|}{n}} \leq \frac{|G|}{\alpha(\Gamma_G)}.
	\]
Thus $\alpha(\Gamma_G) \leq \frac{|G|}{n}$. Since $G$ is
transitive, the set of all vertices in $\Gamma_G$ that correspond to the stabilizer in $G$ of a point achieves this bound.\qed

If $H \leq G$ are both transitive, then the size of a star in $\Gamma_H$ is $|H|/n$, and the size of the star in $G$ is $|G|/n$, which is exactly $|G:H|$ times of the size of a star in $\Gamma_H$. We show that an analogous property carries through to subgraphs of $\Gamma_G$ and $\Gamma_H$.

\begin{lem}
Let $H \leq G\le \sym{n}$ be transitive groups and $D \subseteq \Der(H)$.
If the largest independent set in $\cay(H, \Der(H) \setminus D)$ is a star, then the largest independent set in 
$\cay(G, \Der(G) \setminus D)$ is a star.
\end{lem}
\begin{proof}
Since $H \leq G$, we can map the vertices from $\cay(H, \Der(H) \setminus D)$ to the vertices $\cay(G, \Der(G) \setminus D)$ by embedding. If $h_1$ is adjacent to $h_2$ in  $\cay(H, \Der(H) \setminus D)$, then 
$h_1^{-1}  h_2$ is  a derangement that is not in $D$. Therefore $h_1$ and $h_2$ are also adjacent in $\cay(G, \Der(G) \setminus D)$, so this embedding map is a homomorphism.

So by Lemma~\ref{lem:nohom},
\[
\alpha( \cay(G, \Der(G) \setminus D) ) \leq |G:H| \ \alpha( \cay(H, \Der(H) \setminus D) ).
\]
Since the largest independent set in $\cay(H, \Der(H) \setminus D)$ is a star, this implies
\[
\alpha( \cay(G, \Der(G) \setminus D) ) \leq |G:H| \frac{|H|}{n} = \frac{|G|}{n}
\]
which is the size of a star in $\Gamma_G$.
\end{proof}

In our next result, we show that if $H$ has the EKR property then we cannot increase the independence number in $\Gamma_G$ by removing only derangements that lie outside $H$. This has a similar flavour to some of our earlier results, in which we showed that there exist large sets of labels we can remove without increasing the independence number.

\begin{lem}
Let $H\leq G\le \sym{n}$ be transitive groups and assume that $H$ has the EKR property.
If $D \subseteq G \setminus H$, then 
\[
\alpha ( \cay(G, \Der(G) \setminus D) )  = \frac{|G|}{n}.
\]
In particular, $ \cay(G, \Der(G) \setminus D) $ does not have an independent set larger than a star.
\end{lem}
\begin{proof}
If $D \subseteq G \setminus H$, then there is a homomorphism from
$\Gamma_H$ to $\cay(G, \Der(G) \setminus D).$
(This is because $|D \cap \Der(H) | = 0$, which implies $\Der(H) \subseteq \Der(G) \setminus D$.)
Thus by Lemma~\ref{lem:nohom},
\[
\alpha ( \cay(G, \Der(G) \setminus D) ) \leq \frac{|G|}{|H|} \alpha(\Gamma_H ) = \frac{|G|}{|H|} \frac{|H|}{n} = \frac{|G|}{n}.
\]
which is the size of a star in $\Gamma_G$.
\end{proof}

\subsection{Subgroups with a unique derangement}

This section considers the special case where there is a subgroup $H \leq G \le \sym{n}$ that has only one derangement.
If  $H$ has one single derangement $\sigma$, then $\sigma = \sigma^{-1}$ and so $n$ must be even. In this case the vertices of $H$ induce a subgraph in $\Gamma_G$ that is the union of $|H|/2$ disjoint edges.

\begin{lem}\label{join1}
Let $G$ be a transitive group with $H \leq G$ where $H$ is a group with a unique derangement, $\sigma$.
Then $\alpha(\Gamma_G) \geq |H|/2$. Further, $\alpha(\cay(G, \Der(G)\setminus\{\sigma\} )) \geq |H|$.
\end{lem}
\begin{proof}
Since $H$ has a unique derangement, $\sigma$, the elements of $H$ form an independent set in $\cay( G, \Der(G)\setminus\{\sigma\} )$. In $\Gamma_G$, the induced subgraph on $H$ is a perfect matching.
The subgroup $N =\langle \sigma \rangle$ is a normal subgroup of $H$, and the induced subgraph on any coset of $N$ in $H$ is an edge of the perfect matching. Taking one representative of each coset of $N$ in $H$ forms an independent set in $\Gamma_G$ of size $|H|/2$.
\end{proof}

The \emph{join} of two graphs $X$ and $Y$ has as its vertices the union of vertices in $X$ and vertices in $Y$. Two vertices are adjacent in the join if one of the following holds: they are both in $X$ and adjacent in $X$; they are both in $Y$ and adjacent in $Y$; or one vertex is in $X$ and the other is in $Y$. The join of $X$ and $Y$ is denoted by $X \vee Y$.

Next we give examples of groups with the property that the removal of a single label from the group doubles the size of the largest independent set in the derangement graph. Typically, this situation is the opposite of robustness: this result provides a single label whose removal not only increases the independence number, but doubles it in the same way that producing a binary star would. Unless it happens that $\de_G\le 1$, a group with this property is certainly not EKR robust.

\begin{lem}\label{lem:SubgroupMatching}
Let $G$ be a transitive group of degree $n$ (where $n$ is even) with $H \leq G$ where $H$ is a group with a unique derangement, $\sigma$. If $|G:H| = n/2$ and $H$ has $n/2$ orbits on $\{1,\dots, n \}$, then 
\[
\Gamma_G =  \bigvee_{j=1}^{\frac{n}{2}}  \ \bigcup_{i=1}^{|H|/2} K_2.
\]
In this case $\alpha(\Gamma_G) = |H|/2$ and $\alpha(\cay(G,  \Der(G) \setminus \{ \sigma \} )) = |H|$ and $G$ has the EKR property.
\end{lem}

\begin{proof}
First, the subgraph of $\Gamma_G$ induced by $H$ is a perfect matching with $|H|/2$ edges, that is a copy of $\displaystyle{\bigcup_{i=1}^{|H|/2} K_2}$. The same is true for the subgraph induced by any coset of $H$.

Using Burnside's lemma, with the fact that $G$ is transitive, $1 =\frac{1}{|G|}  \sum_{g \in G} \fix(g)$.
By assumption, $H$ has $n/2$ orbits on $\{1,2, \dots, n\}$, again by Burnside's lemma
$\frac{n}{2} |H|  =  \sum_{h \in H} \fix(h)$. Since $|G:H| = \frac{n}{2}$, 
\begin{align*}
1 &= \frac{1}{|G|} \sum_{g \in G} \fix(g)  \\
& = \frac{1}{|G|} \left(  \sum_{h \in H} \fix(h)  + \sum_{g \in G \setminus H} \fix(g)\right ) \\
&= \frac{n|H|}{2 |G|}  + \frac{1}{|G|}  \sum_{g \in G \setminus H} \fix(g) \\
&= 1 + \frac{1}{|G|}  \sum_{g \in G \setminus H} \fix(g).
\end{align*}
This implies that $\fix(g) = 0$ for every $g \in G\setminus H$, so every element in $G$ that is not in $H$ is a derangement. In this case, every $h \in H$ is adjacent to every vertex in $G \setminus H$, in particular there is an edge between every pair of vertices from different cosets of $H$. Thus $\Gamma_G$ is the join of $|G:H| = \frac{n}{2}$ copies of $\cup_{i=1}^{|H|/2} K_2$. It is clear from the structure of this graph that the size of the maximum independent set in $\Gamma_G$ is $|H|/2$. Since $|G:H| = n/2$, we have that $\frac{1}{2} |H| = \frac{1}{2} 2|G|/n = |G|/n$.

Finally, the subgraph induced by $H$ in $\cay(G,  \Der(G) \setminus \{ \sigma \} )$ is an empty graph on $H$ vertices. Similar to above $\alpha(\cay(G,  \Der(G) \setminus \{ \sigma \} )) = |H|$.
\end{proof}


While the conditions of Lemma~\ref{lem:SubgroupMatching} seem very specific, there are a surprising number of groups do that satisfy them.  For example, it is straight-forward to construct such a group with degree $2n$. Define $x_i = (2i-1, 2i)$, and let $H$ be the group generated by $x_1, \dots, x_n$. Define $y = (1,3,\dots, 2n-1)(2,4,\dots, 2n)$, and let $G$ be the group generated by $H$ and $y$. Then $G$ satisfies Lemma~\ref{lem:SubgroupMatching}. 

It can be seen from the lemma, that derangement graph for group satisfying Lemma~\ref{lem:SubgroupMatching} is the join of copies of unions of disjoint edges. The eigenvalues of such a graph are
\[
\left \{ \frac{n-2}{2} |H| +1, \, 1 , \,  -1, \,  -(|H| + 1) \right \}
\]
(any graph on $\frac{n}{2} |H|$ vertices with these eigenvalues will be the join of copies unions of disjoint edges).
So it is easy to determine if a derangement graph has this form by simply checking its eigenvalues. 
Using GAP~\cite{GAP4} we can determine how many groups have this form among all the transitive groups of a given degree.

\begin{table}[h]
\begin{tabular}{|c|c|c|} \hline 
 Degree & Number of groups  & Total number of \\
 & satisfying Lemma~\ref{lem:SubgroupMatching}  & transitive groups  \\ \hline 
4 & 1  & 5\\
6 & 2  & 16\\
8 & 12  & 50 \\
10 & 2  & 45\\
12 & 21 & 301 \\
14 & 3  & 63\\
16 & $\geq$ 167 & 1954\\
18 & 13 & 983\\
20 & 23 & 1117 \\ \hline 
\end{tabular}
\caption{Number of Groups satisfying Lemma~\ref{lem:SubgroupMatching}  }
\end{table}


\section{Open problems and future work}\label{sect:open}

We have shown in this paper that regular representations and Frobenius representations of groups
have the property that a single label can be removed from the connection set, to produce a Cayley graph whose independence number is twice that of the original derangement graph. 
However, after that initial jump, it is often possible to remove many more labels before further increasing the independence number. We have observed that this covers all transitive representations of abelian and generalised dicyclic groups. We have also shown if the index-2 abelian subgroup of a generalised dihedral group includes any odd permutations, then an analogous result holds.

\begin{question}
Are there other transitive group representations that have similar properties---either with respect to a class of representations (as for regular and Frobenius representations) or with respect to all transitive representations of some family of groups (like the generalised dihedral groups over abelian groups that include odd permutations)?
\end{question}

We introduced the definition of EKR robust groups. Intuitively, groups that are \emph{not} EKR robust are those in which relatively few labels can be removed from the connection set of the derangement graph to produce a subgraph with a larger independence number. In this paper we included a small number of groups and constructions that produce groups that are not EKR robust. It would be interesting to find larger and more complex families that are not EKR robust.

\begin{question}
Find examples of groups that are far from being EKR robust, that is groups $G$ with 
\[
\alpha( \cay(G, \Der(G) \setminus D) ) > \alpha(\Gamma_G)
\]
where $D$ has fewer labels than $D_{i \to j} \cup D_{j \to i}$.
\end{question}

On the other end of the spectrum, we have shown the group $\pgl(2,q)$ is EKR robust in its 2-transitive action, and for $\sym{n}$ in its natural permutation representation, a large number of labels must be removed before the independence number increases. Computer search shows $\sym{n}$ is robust for $n=3$ and 4, but not for $n=5$. We could not generalize this to larger symmetric groups, nor find other examples. This is partly due to the fact that the derangement graphs on larger symmetric groups are not easily susceptible to computer searches. But, we suspect $n=5$ is anomalous among the symmetric groups and in general at least $\de_{\sym{n}}$ labels need to be removed from the connection set of the derangement graph before a larger independent set is created.

\begin{question}
For $n$ sufficiently large, is $\sym{n}$ EKR robust? Specifically, is there an inverse-closed set $D  \subset  \Der(n)  $ such that $D$ has fewer labels than $D_{i \to j}\cup D_{j \to i}$ and 
\[
\alpha(\cay(\sym{n}, \Der(n)) < \alpha( \cay(\sym{n}, \Der(n) \setminus D) )?
\] 
\end{question}

If there is a set $D$ as in the previous question, then there can be no regular subgroups contained in 
$( \Der(n) \setminus D') \cup \{ \id \}$; so a related question is to determine the minimal number of elements that need to be removed from $\Der(n)$ to achieve this.

\begin{question}
For $n>5$, is there a way to remove derangements from $\Der(n)$ so that the resulting set, together with $\id$, has no transitive subgroups, without removing all the derangements that map $i$ to $j$?
\end{question}

We can also ask the same questions for the alternating group.

\begin{question}
For $n$ sufficiently large, is $\alt{n}$ EKR robust?
\end{question}

In general we can ask what other groups, or families of groups, are robust, or, more generally, what group are ``close" to being robust.

\begin{question}
Find more (and interesting) examples of groups that are EKR robust. Find examples of groups $G$ such that $\alpha( \Gamma_G ) = \alpha ( \cay(G, Der(G) \setminus D) )$, unless $D$ is large, in some sense? 
 \end{question}

\begin{question}
Can anything be determined about the EKR robustness of different families of groups, like nilpotent groups, solvable groups, or other specific families of groups?
\end{question}

It is known that all 2-transitive groups have the EKR property~\cite{meagher2016erdHos}. Since not all 2-transitive groups are EKR robust (for example $\sym{5}$ is not), EKR robustness may be an effective way to differentiate among groups that have the EKR property. In fact, robustness may be a way to further classify groups that have the EKR property, similar to how intersection density classifies the groups that do not have the EKR property (see~\cite{li2020erd}).

Following the notion of the strict EKR property, we can define a permutation group $G$ to be strict EKR robust if $ \alpha( \cay(G, \Der(G) \setminus D ) ) =  \frac{|G|}{n}$, unless 
$D_{i\to j} \cup D_{j \to i} \subseteq D$ for some $i, j$. Any Frobenius group $G$ has $\Gamma_G = K_n \cup K_n$, so removing a single derangement (and its inverse) from $\Der(G)$ results in a Cayley graph with a larger independent set. These groups are all strict EKR robust (in a trivial sense) since any derangement that maps $i$ to $j$ is the entire set $D_{i \to j}$. With slightly different details (the graph is just $K_n$), the same is true for any regular representation. It is open if there are other, more interesting, group actions that are strict EKR robust. However, we will make the following conjecture.

\begin{conj}
The group $\pgl(2,q)$ with its natural $2$-transitive action is strict EKR robust.
\end{conj}

Finally, we note that in this paper we only consider derangement graphs, but this problem can be considered for any Cayley graph.

\bibliographystyle{plain}

\end{document}